\documentclass[11pt,a4paper]{amsart}
\usepackage[english]{babel}
\usepackage[T1]{fontenc}
\usepackage[utf8]{inputenc}

\usepackage{amsthm}
\usepackage{amsmath}
\usepackage{amssymb}
\usepackage{mathrsfs}
\usepackage{mathtools}
\usepackage{stmaryrd}
\usepackage{csquotes}
\usepackage{mathptmx}
\usepackage{shuffle}
\usepackage{bm} 
\usepackage{comment}
\usepackage{physics}

\usepackage[section]{placeins}

\usepackage{graphicx}
\usepackage{hyperref}
\usepackage{lmodern}
\usepackage{fullpage}
\usepackage{fancybox}
\usepackage{xcolor}
\usepackage{relsize}
\usepackage{ifthen}
\usepackage{tikz}
\usepackage{tikz-cd}
\usetikzlibrary{matrix}

\usepackage[left=2cm, right=2cm, top=3cm, bottom=3cm]{geometry}

\newcommand\bb{\mathbb}

\DeclareMathOperator{\GL}{GL}
\DeclareMathOperator{\Id}{Id}

\newcommand\dif{\text{dif}}
\newcommand\inv{\text{inv}}

\newcommand\boxcon{{~\boxed{\!*\!}~}}
\newcommand\smboxcon{{\boxed{\!\!*\!\!}}}
\newcommand\boxconv{{~\boxed{\!*\!}|~}}
\newcommand\boxconvred{{~\overline{\boxed{\!*\!}}~}}
\newcommand\boxconvredred{{~|\boxed{\!*\!}|~}}
\newcommand\boxconvline{{~{|\boxed{\!*\!}}~}}

\newcommand\Mult{\text{Mult}}

\newcommand\sqdot{\mathbin{\vcenter{\hbox{\rule{.5ex}{.5ex}}}}}

\newcommand\boxsemitimes{{\hspace{2pt}\begin{tikzpicture}[scale=.24] \draw (0,0)--(1,0)--(1,1)--(0,1)--(0,0)--(1,1); \end{tikzpicture}\hspace{2pt}}}
\newcommand\boxsemiplus{{\hspace{2pt}\begin{tikzpicture}[scale=.24] \draw (0,0)--(1,0)--(1,1)--(0,1)--(0,0) (.5,0)--(.5,1) (.5,.5)--(1,.5); \end{tikzpicture}\hspace{2pt}}}
\newcommand\starbox{{\hspace{2pt}\begin{tikzpicture}[scale=.24] \draw (0,0)--(1,0)--(1,1)--(0,1)--cycle (.3,.3)--(.7,.7) (.3,.7)--(.7,.3) (.2,.5)--(.8,.5) (.5,.2)--(.5,.8); \end{tikzpicture}\hspace{2pt}}}
\newcommand\smboxsemitimes{{\begin{tikzpicture}[scale=.15] \draw[thin] (0,0)--(1,0)--(1,1)--(0,1)--(0,0)--(1,1); \end{tikzpicture}}}
\newcommand\smboxsemiplus{{\begin{tikzpicture}[scale=.15] \draw (0,0)--(1,0)--(1,1)--(0,1)--(0,0) (.5,0)--(.5,1) (.5,.5)--(1,.5); \end{tikzpicture}}}
\newcommand\smstarbox{{\begin{tikzpicture}[scale=.15] \draw (0,0)--(1,0)--(1,1)--(0,1)--cycle (.3,.3)--(.7,.7) (.3,.7)--(.7,.3) (.2,.5)--(.8,.5) (.5,.2)--(.5,.8); \end{tikzpicture}}}

\theoremstyle{plain}
\newtheorem{thm}{Theorem}[section]
\newtheorem{prop}[thm]{Proposition}
\newtheorem{lem}[thm]{Lemma}
\newtheorem{coroll}[thm]{Corollary}

\theoremstyle{definition}
\newtheorem{defn}[thm]{Definition}

\theoremstyle{remark}
\newtheorem{rmk}[thm]{Remark}

\title{A post-group theoretic perspective on the\\ operator-valued S-transform in free probability}

\author[Kurusch Ebrahimi-Fard]{Kurusch Ebrahimi-Fard${}^{\diamond}$}
\address[${}^{\diamond}$]{Department of Mathematical Sciences, Norwegian University of Science and Technology (NTNU), 7491 Trondheim, Norway. Centre for Advanced Study (CAS), 0271 Oslo, Norway.}
\email{kurusch.ebrahimi-fard@ntnu.no}
\urladdr{https://folk.ntnu.no/kurusche/}

\author[Timoth\'e Ringeard]{Timoth\'e Ringeard${}^{\dagger}$}
\address[${}^{\dagger}$]{\'Ecole Normale Sup\'erieure, 75005 Paris, France.}
\email{timothe.ringeard@ens.psl.eu}

\subjclass[2020]{46L54, 22E60, 17B38, 16W60}
\keywords{operator-valued free probability, R-transform, S-transform, post-group, post-Lie algebra, pre-Lie algebra, crossed morphism, Nijenhuis operator, relative Rota--Baxter operator}

\date{\today}

\begin{document}

\maketitle

\begin{abstract}
We investigate the algebraic structure underlying Voiculescu's S-transform in the setting of operator-valued free probability. We show that its twisted factorisation property gives rise to post-groups, crossed morphisms, as well as pre- and post-Lie algebras.
\end{abstract}

\tableofcontents


\section{Introduction}
\label{sec:intro}

Noncommutative probability theory extends classical probability by accommodating the case of random variables which do not necessarily commute. It provides an approach for exploring noncommutative probability distributions. A prime example is Voiculescu's free probability theory \cite{VDN1992}, where the concept of freeness replaces the classical notion of independence. In \cite{speicher1994multiplicative}, Speicher showed that the combinatorial aspects underlying the relationships among moments and cumulants in free probability are intimately connected to noncrossing set partitions. 

The R- and S-transforms are formal power series that play an integral part in free probability, offering a mathematical framework for studying the relationships among free random variables \cite{nica_speicher_book}. Starting from a noncommutative probability space $(\mathcal{A}, \phi)$, which consists of an unital linear functional $\phi: \mathcal{A} \to \mathbb{K}$ mapping an unital associative algebra $\mathcal{A}$ into the base field $\mathbb{K}$ ($\phi(1_\mathcal{A} )=1$), and assuming that the random variable $a \in \mathcal{A}$ is such that $\phi(a)=1$, then the S- and R-transforms, $S_a(z) \in 1 + z\mathbb{K}\langle\langle z \rangle\rangle$ respectively $R_a(z) \in z+z^2\mathbb{K}\langle\langle z \rangle\rangle$, are related through
\allowdisplaybreaks
\begin{equation}
\label{scalarRSrel}
	zS_a(z)=R_a(z)^{\circ -1}.
\end{equation}
The inverse of the R-transform on the right-hand side of \eqref{scalarRSrel} is defined with respect to the standard composition of formal power series within $z+z^2\mathbb{K}\langle\langle z \rangle\rangle$. A third formal power series is the moment series, $M_a(z)= \sum_{n>0}\phi(a^n)z^n \in z+z^2\mathbb{K}\langle\langle z \rangle\rangle$, which relates to the S- respectively R-transform as follows
\begin{equation}
\label{scalarmoments}
	zS_a(z)=(1+z) M_a(z)^{\circ -1}.
\end{equation}

Both the S- and R-transforms exhibit specific properties concerning the addition ($a+b \in \mathcal{A}$) and multiplication ($ab \in \mathcal{A}$) of two free random variables. Indeed, the R-transform linearises the former
\begin{equation}
\label{scalarR}
	R_{a+b}(z) = R_{a}(z) + R_{b}(z)
\end{equation}
whereas the S-transform factorises the latter
\begin{equation}
\label{scalarS}
	S_{ab}(z) = S_{a}(z) S_{b}(z). 
\end{equation}
The product on the righthand side of \eqref{scalarS} is the commutative Cauchy product on $\mathbb{K}\langle\langle z \rangle\rangle$.  

Free cumulants are the coefficients of the R-transform of $a \in \mathcal{A}$, $R_{a}(z) = \sum_{n>0} k_n^az^n$. Using that freeness of random variables is equivalent to the vanishing of mixed free cumulants \cite[Theorem 11.20, p.186]{nica_speicher_book}, the coefficients of the R-transform satisfy for a product $ab$ of free random variables the well-known formula \cite[Theorem 14.4, p.226]{nica_speicher_book}  
\begin{equation}
\label{scalarcumulants}
	k_{n}^{ab} = \sum_{\pi \in \mathrm{NC}(n)} k_{\pi}^{a} k_{{\rm Kr}(\pi)}^{b}.
\end{equation}
Here the sum runs over the set $\mathrm{NC}(n)$ of noncrossing partitions of order $n$ and ${\rm Kr}(\pi) \in \mathrm{NC}(n)$ is the so-called Kreweras complement of the noncrossing partition $\pi \in \mathrm{NC}(n)$ \cite[Definition 9.21, p.147]{nica_speicher_book}.

Following \cite{nica_speicher_book} -- to which we refer the reader for details, -- identity \eqref{scalarcumulants} is extended to the level of formal power series in terms of the so-called boxed convolution product \cite[Definition 17.1, p.273]{nica_speicher_book}
$$
	R_{ab}(z) = (R_{a}  \boxcon R_b)(z),  
$$
which defines a group law on $z + z^2 \mathbb{K}\langle\langle z \rangle\rangle$. In Theorem 18.14 and Corollary 18.17, Nica and Speicher \cite[p.294]{nica_speicher_book} deduce identity \eqref{scalarS} by showing that relation \eqref{scalarRSrel} codes a group isomorphism. Indeed, Theorem 18.14 identifies the map $\mathcal{F}$ from the group $(z + z^2 \mathbb{K}\langle\langle z \rangle\rangle,\boxcon )$ to the group $1 + \mathbb{K}\langle\langle z \rangle\rangle$ with commutative Cauchy product,  defined such that 
\begin{equation}
\label{eq:S_transf_as_F_map}
	\mathcal{F}(R_a)(z)=S_a(z),
\end{equation} 
as a group isomorphism, that is
\begin{equation}
\label{eq:groupmorph}
	\mathcal{F}(R_{a}  \boxcon R_b)(z)=S_a(z)S_b(z).
\end{equation} 

\smallskip

Voiculescu extended free probability to the realm of noncommutative operator-valued random variables \cite{voiculescu1995}. Following Dykema \cite{dykema2005strans}, the operator-valued counterparts of the R- and S-transforms are now formal multilinear function series defined over an operator-valued probability space \cite{MingoSpeicher2017}. In this framework, Dykema generalised the scalar-valued identity \eqref{scalarS} by showing that the operator-valued S-transform satisfies a twisted factorisation with respect to the product $ab$ of two free random variables
\begin{equation}
\label{operatorS}
	S_{ab} = S_{b} \cdot S_{a}\circ (S_{b}^{-1}\cdot  I \cdot  S_{b}).
\end{equation}

In \cite{part1}, we approach \eqref{scalarcumulants} and \eqref{operatorS} in the operator-valued setting using planar binary trees and their combinatorial properties instead of noncrossing partitions. The proofs of both operator-valued identities are presented through the utilisation of four convolution-type products defined explicitly on formal multilinear function series in terms of combinatorial operations on planar binary trees \cite[Definition 3.16]{part1}. 

\smallskip

In the work at hand, we revisit Nica's and Speicher's Theorem 18.14 and Corollary 18.17 \cite[p.294]{nica_speicher_book}, by focusing on the distinctive structure of the twisted factorisation property \eqref{operatorS}. Unveiling the algebraic intricacies of the right-hand side of \eqref{operatorS} which combines multiplication and composition of formal series of multilinear functions, we describe several interconnected post-groups \cite{al-kaabi_ebrahimi-fard_manchon,Bai_Guo_Sheng_Tang_post_groups,KIA2023}, as well as post-Lie algebras, the corresponding infinitesimal objects. Additionally, we offer insights into the relations among the involved post-groups through the use of the notion of crossed morphisms \cite{KIA2023,hilgert-neeb2010}.

\medskip

An outline of the paper is presented now. In Section \ref{sec:op-valued_proba}, we recall from Dykema \cite{dykema2005strans} the construction of two binary products, composition and multiplication, on formal multilinear function series, as well as several basic facts about operator-valued probability spaces. We describe how multilinear function series can be used to generalise the moment and cumulant series (or R-transform) associated to a random variable in an operator-valued probability space. 

In Section \ref{sec:group_str}, we describe group laws over subsets of multilinear function series. The starting point are the generalisations to the operator-valued framework of the Cauchy product defined over $1+z\bb K\langle\langle z\rangle\rangle$, and the composition product defined over $z+z^2\bb K\langle\langle z\rangle\rangle$. We show that noncommutativity of the base algebra in the operator-valued setting allows for the definition of right- and left-products, $\star_r$ and $\star_l$. From this, we motivate the introduction of the S-transform as an inverse map. We also define another product, denoted $\sqdot\, $, which is closely related to the twisted factorisation \eqref{operatorS}, and therefore to the operator-valued boxed convolution product $\boxcon$. Moreover, most of these new products fall under the formalism of post-groups. From this point of view, we describe many links between the products $\star_r$, $\star_l$, $\sqdot\, $, and $\boxcon$, and we revisit the free moment-cumulant relations. Finally, we draw parallels between some particular series, $H_1$ and $H_2$, that we deduce from the post-group formalism, and Lenczewski's s-free additive and multiplicative convolutions, known to play important roles in the operational approach to subordination. 

In Section \ref{sec:lie_structure}, we transfer the post-group structures to the Lie algebra level, taking inspiration from the works~\cite{al-kaabi_ebrahimi-fard_manchon,Bai_Guo_Sheng_Tang_post_groups,KIA2023}. This results in post-Lie algebras defined again over a subset of formal multilinear function series. We show that the post-Lie product originating from the product $\sqdot$ is the sum of the two post-Lie products originating from the products $\star_r$ and $\star_l$, and we elaborate on the reasons of why this is the case. We also describe a pre-Lie product over the same set, defined as a linearisation of composition. In that setting, we identify left- and right-shifting operators satisfying the associative Nijenhuis relation, and that can be considered as generalisations of $f(z)\in \bb K\langle\langle z\rangle\rangle \mapsto zf(z)$.

\medskip

{\bf{Acknowledgements}}: TR would like to thanks the Department of Mathematical Sciences at the Norwegian University of Science and Technology in Trondheim for hospitality. KEF is supported by the Research Council of Norway through project 302831 ''Computational Dynamics and Stochastics on Manifolds" (CODYSMA). He would also like to thank the Centre for Advanced Study (CAS) in Oslo for hospitality and its support during the academic year 2023/2024.


\section{Operator-valued probability}
\label{sec:op-valued_proba}


\subsection{Formal series of multilinear functions}
\label{ssec:series_multilin_functions}

We follow the terminology of Dykema in \cite{dykema2005strans}. Throughout the remainder of the paper, the base field of characteristic zero, over which all algebraic structures are defined, is denoted by $\mathbb{K}$.

\begin{defn}
Let $B$ be a $\mathbb{K}$-algebra with unit $1$. A sequence of multilinear functions $f = (f_n)_{n\ge 0}$, $f_n : B^{\otimes n} \to B$ (with $B^{\otimes 0} = \mathbb{K}$) will be called a formal multilinear function series. The set of these formal multilinear function series will be called $\Mult [[B]]$. 
\begin{enumerate}

\item[i)]
We can multiply such series by defining for $f,g \in \Mult[[B]]$:
\begin{equation}
\label{multipl}
	(f \cdot g)_n(x_1, \ldots, x_n) := f_n(x_1, \ldots, x_n)g_0 + f_0g_n(x_1, \ldots, x_n) 
	+ \sum_{k=1}^{n-1} f_k(x_1, \ldots, x_k) g_{n-k}(x_{k+1}, \ldots, x_n).
\end{equation}
$(\Mult[[B]], \cdot)$ is a monoid, with unit $1 := (\delta_{n,0}1)_{n\ge 0}$. We will often write $fg$ for $f\cdot g \in \Mult[[B]]$. 

\item[ii)]
We can compose such series by defining for $f,g \in \Mult[[B]]$, such that $g_0=0$:
\begin{equation}
\label{compo}
	(f\circ g)_n(x_1, \ldots, x_n) 
	:= \sum_{\substack{n = k_1 + \cdots + k_l \\ k_i\ge 1,\ l \ge 0}} f_l(g_{k_1}(x_1, \ldots, x_{k_1}), \ldots, g_{k_l}(x_{n-k_l+1}, \ldots, x_n)).
\end{equation}
The series $I := (\delta_{n,1}\Id)_{n\ge 0}$ is a unit on the left and on the right for this product.
\end{enumerate}
Note that $f\cdot g = 0$ implies either $f=0$ or $g=0$. This is, however, not the case for composition \eqref{compo}. Regarding expressions  in $\Mult[[B]]$ involving both composition and multiplication, $\circ$ and $\cdot$, the former operation will be given precedence, i.e., $f \cdot g \circ h \cdot k = f\cdot (g\circ h) \cdot k$.
\end{defn}

We define now groups, $(G_B^\inv ,\cdot )$ and  $(G_B^\dif , \circ)$, with respect to multiplication \eqref{multipl} respectively composition \eqref{compo} of elements in $\Mult[[B]]$:
\begin{align*}
	G_B^\inv &:= \{f \in \Mult [[B]] \mid f_0\in B^\times\} \\
	G_B^\dif &:= \{f \in \Mult [[B]] \mid f_0=0, f_1\in \GL(B)\},
\end{align*}
where $B^\times$ is the set of invertible elements of $B$. We note that replacing $f_0\in B^\times$ by $f_0=1$ and $f_1\in \GL(B)$ by $f_1=\Id$ would not change much. For $f \in G_B^\inv$, we will denote by $f^{-1}$ the multiplicative inverse of $f$. The compositional inverse of $f \in G_B^\dif$ will be denoted by $f^{\circ-1}$. 

The set $G_B^\inv $ can be mapped to a subset of $G_B^\dif$ by multiplication on the left or on the right by the composition unit $I$. It gives rise to two more sets:
\begin{align}
\label{Ishiftset}
	G_B^{I,l}:= I \cdot G_B^\inv \\ 
	G_B^{I,r}:= G_B^\inv \cdot I. 
\end{align}
Alternatively,  
\begin{align*}
	G_B^{I,l} := \{f \in G_B^\dif \mid \forall n, f_n(x_1, x_2, \ldots, x_{n-1}, x_n) = x_1 f_n(1, x_2, \ldots, x_{n-1}, x_n)\} \\ 
	G_B^{I,r} := \{f \in G_B^\dif \mid \forall n, f_n(x_1, x_2, \ldots, x_{n-1}, x_n) = f_n(x_1, x_2, \ldots, x_{n-1}, 1) x_n\}.
\end{align*}
Note that in \cite{part1} the set $G_B^{I,l}$ was denoted by $G_B^{I}$, whereas the set $G_B^{I,r}$ was not explicitly defined.

\begin{lem} \cite{dykema2005strans}
\label{lem:rightaction}
Composition and multiplication, defined in \eqref{compo} respectively \eqref{multipl}, are both associative. Composition is noncommutative, and multiplication is noncommutative as soon as $B$ is noncommutative. Moreover, composition is right-distributive over multiplication, i.e., for any $f,g,h \in \mathrm{Mult}[[B]]$, $h_0=0$, we have that
\begin{equation}
\label{rightaction}
	(f \cdot g)\circ h = (f\circ h) \cdot (g\circ h).
\end{equation}
\end{lem}

\begin{proof}
See Dykema \cite{dykema2005strans} for details. 
\end{proof}

\begin{lem}
\label{lem:G_B^I_group}
The sets $G_B^{I,l}$ and $G_B^{I,r}$ form groups with respect to composition \eqref{compo}.
\end{lem}

\begin{proof}
See \cite{part1} for the prove regarding $G_B^{I,l}$. The case of $G_B^{I,r}$ follows a similar argument.
\end{proof}

\begin{defn} (S-transform)
\label{def:S_transform}
Let $F \in G_B^\inv$. Define its left respectively right S-transforms by
\begin{align*}
	I \cdot S^l(F) &:= (IF)^{\circ-1} \\ 
	S^r(F) \cdot I &:= (FI)^{\circ-1}.
\end{align*}
\end{defn}

\begin{lem}
\label{lem:fixed_point_S_f}
Let $F \in G_B^\mathrm{inv}$. Its left and right S-transforms are the unique solutions to the fixed point equations:
\begin{align}
\label{eq:fixed_point_S_f}
\begin{aligned}
	S^l(F) = F^{-1} \circ (I \cdot S^l(F)) \\ 
	S^r(F) = F^{-1} \circ (S^r(F) \cdot I).
\end{aligned}
\end{align}
\end{lem}

\begin{proof}
See \cite{part1}.
\end{proof}

We note that the foregoing lemma also provides a formula for the product $S^l(F)^{-1} \cdot I \cdot S^l(F)$ (which will appear later).

\begin{coroll}
\label{cor:other_expr_S_f-1IS_f}
Let $F \in G_B^\mathrm{inv}$. Then 
$$
	S^l(F)^{-1} \cdot I \cdot S^l(F) = (FI)\circ (IF)^{\circ -1}.
$$
\end{coroll}


\subsection{Operator-valued probability space}
\label{ssec:ovps}

Suppose $B$ is a unital algebra over the field $\mathbb{K}$. We recall some definitions about operator-valued free probability. See \cite{MingoSpeicher2017} for more details and background.

\begin{defn}\cite{MingoSpeicher2017}
\label{def:op_valued_free_proba}
\begin{enumerate}
	\item A \textit{unital $B$-algebra} is an unital $\mathbb{K}$-algebra $\mathcal A$, that is a bi-module over $B$ such that for all $a_1,a_2\in \mathcal A$ and $b_1, b_2\in B$, 
	\begin{align*}
		(b_1 \cdot a_1) \cdot b_2 &= b_1 \cdot (a_1 \cdot b_2)\\
		b_1 \cdot (a_1 \cdot a_2) &= (b_1 \cdot a_1) \cdot a_2 \\
		(a_1 \cdot a_2) \cdot b_1 &= a_1 \cdot (a_2 \cdot b_1) \\
		(a_1\cdot b_1)\cdot a_2 &= a_1 \cdot (b_1 \cdot a_2).
	\end{align*}
	By a slight abuse of notation, all products and actions are denoted by the same symbol $\cdot$. In other words, we ask that any product $x_1x_2x_3 = x_1\cdot x_2 \cdot x_3$ is uniquely defined, for $x_1, x_2, x_3\in B \cup \mathcal A$. Note that if $B$ is a sub-algebra of $\mathcal A$, then $\mathcal A$ is automatically a $B$-algebra.
	
	\item A \textit{$B$-valued probability} space is a pair $(\mathcal A, \bb E)$, where $\mathcal A$ is an unital $B$-algebra, and $\bb E : \mathcal A \to B$ is a unital linear map, $\bb E(1_{\mathcal A}) = 1$, such that for all $x, y \in B$ and for all $a \in \mathcal A$, $\bb E$ is what we call $B$-balanced
	\begin{equation}
	\label{eq:Bbimod}
		\bb E(xay) = x \bb E(a) y.
	\end{equation}

	Elements $a \in \mathcal A$ are called \textit{(noncommutative) random variables}. 
	Their distribution is usually given by their \textit{moment series} $M_a \in \Mult[[B]]$ defined by
	$$
		M_a(x_1, \ldots, x_n) = \bb E(a x_1 a x_2 a \cdots a x_n a).
	$$
	
	\item (Freeness)
	A collection of $B$-subalgebras $(\mathcal A_j)_{j \in J}$ of $\mathcal A$ is said to be \textit{free} if $\bb E(a_1 \cdots a_n) = 0$ when the following conditions hold:
	\begin{itemize}
		\item $n\ge 1$,
		\item $\bb E(a_i) = 0$ for $1\le i\le n$,
		\item $a_i \in \mathcal A_{j_i}$ for $1 \le i \le n$,
		\item $j_i \neq j_{i+1}$ for $1 \le i < n$.
	\end{itemize}
	A collection of elements $(a_j)_{j \in J} \subseteq \mathcal A$ is said to be free if their respective generated $B$-subalgebras are free. In this paper, for simplicity, we shall focus on the case of two free elements $a,b \in \mathcal A$.

\end{enumerate}
\end{defn}

Let us fix a $B$-valued probability space $(\mathcal A, \bb E)$.

\begin{defn}
\label{def:op_valued_free_cumulants}
\begin{enumerate}
	\item (Free cumulants)
	For $a \in \mathcal A$, we define its free cumulant series $K_a \in G_B^\mathrm{inv}$
	\[
		I\cdot K_a = (IM_a)\circ (I + IM_aI)^{\circ-1}.
	\]
	Note that althought the definition is different, it coincides with the definition from \cite{part1}.
	\item (S-transform of a random variable)
	Let $a\in \mathcal A$. Define its S-transform $S_a \in G_B^\mathrm{inv}$ by
	\[
		S_a = S^l(K_a).
	\]
\end{enumerate}
\end{defn}

\begin{rmk}
For $a \in \mathcal A$, we can find back the moment series from the free cumulant series, since
\[
	I + IM_aI = ((1+IK_a)^{-1}I)^{\circ-1}.
\]
See Lemma \ref{lem:equiv_MC_relations}, or \cite[Lemma 6.9]{dykema2005strans}, for more details. Therefore the free cumulant series $K_a$ characterises the distribution of $a$, and so does its S-transform $S_a$.
\end{rmk}

The following theorems characterise the distribution of certain combinations of elements in $\mathcal A$.

\begin{thm}\cite{dykema2005strans}
\label{thm:cumulants_free_sum}
Let $a,b \in \mathcal A$ be free random variables. Then
\[
	K_{a+b} = K_a + K_b.
\]
\end{thm}

\begin{thm} \cite{dykema2005strans,T-transf,speicher_short_proof}
\label{thm:S-transf_twisted_factorisation}
Let $a,b\in \mathcal A$ be free random variables, such that both $\bb E(a)$ and $\bb E(b)$ are invertible in $B$. Then
\begin{equation}
\label{eq:S-transf_twisted_factorisation}
	S_{ab} = S_b \cdot S_a \circ (S_b^{-1} \cdot I \cdot S_b).
\end{equation}
\end{thm}



\section{Group and post-group structures on multilinear function series}
\label{sec:group_str}


\subsection{Groups laws on \texorpdfstring{$G_B^\inv$}{G B inv}}
\label{ssec:group_laws}

Theorem \ref{thm:S-transf_twisted_factorisation} presents a combination of elements in $\Mult[[B]]$, which is not obvious -- in particular, if compared to the case of commutative $B$. Therefore, let us precise this expression and show that the righthand side of \eqref{eq:S-transf_twisted_factorisation} actually carries a group structure, when restricting to elements in $G_B^\inv$. Let us first define two group laws on $G_B^\inv$ that arise naturally from the composition product.

\begin{defn}
\label{def:star_l_star_r}
For $F,G\in G_B^\inv$, define
\begin{align}
\label{eq:starproducts}
\begin{aligned}
	F \star_l G := G \cdot F\circ(IG) \\ 
	F \star_r G := F \circ(GI) \cdot G.
\end{aligned}	
\end{align}
\end{defn}

\begin{rmk}
\begin{enumerate}
\item[i)]
It is interesting to compare the products in Definition \ref{def:star_l_star_r} with the particular shifted substitution product defined on noncommutative multivariate power series in \cite[Definition 2.1]{sigma2023}, which has its origin in the Hopf algebraic formulation of monotone additive convolution \cite{EFP2018}. In the scalar-valued case, i.e., when $B=\mathbb{K}$, the products \eqref{eq:starproducts} reduce to the shifted substitution product on $\mathbb{K}\langle\langle z \rangle\rangle$
\begin{equation}
\label{eq:shiftedsub}
	(f \star g)(z)=g(z)f(zg(z)), 
\end{equation}
for $f,g \in \mathbb{K}\langle\langle z \rangle\rangle$. Note that in reference \cite{sigma2023}, the more general multivariate case is considered. 

\item[ii)]
We remark that the link to operator-valued monotone independence is made explicit by Hasebe and Saigo in \cite{HasebeSaigo2014}. Indeed, the second product in \eqref{eq:starproducts} appears explicitly in \cite[Definition 4.4]{HasebeSaigo2014} and it is shown to be associative.
\end{enumerate}
\end{rmk}

\begin{lem}
\label{lem:star_l_r_group_laws}
Both operations, $\star_l$ and $\star_r$, define group laws on $G_B^\mathrm{inv}$, with unit $1$.
\end{lem}

\begin{proof}
Let us consider the bijections
\begin{align}
\label{eq:leftrightmaps}
\begin{aligned}
\begin{array}{rccc}
	\lambda : 	& G_B^\inv 	& \to 		& G_B^{I,l} \\ 
			& F 			& \mapsto & \lambda(F):=IF 
\end{array}
 \qquad\
\begin{array}{rccc}
	\rho : & G_B^\inv 	& \to 		& G_B^{I,r} \\ 
		 & F 			& \mapsto & \rho(F):=FI
\end{array}.
\end{aligned}	
\end{align}

From Lemma \ref{lem:G_B^I_group}, we know that $(G_B^{I,l}, \circ)$ and $(G_B^{I,r},\circ)$ are both groups with unit $I$. Transferring back the structure to $G_B^\mathrm{inv}$, we get the laws $\star_l$ and $\star_r$, that are therefore group laws with unit $1 = \lambda^{-1}(I) = \rho^{-1}(I)$. Indeed, for $F,G \in G_B^\mathrm{inv}$,
\begin{align*}
	\lambda(F) \circ \lambda(G) = (IF) \circ (IG) 
		&= IG \cdot F\circ(IG) = I (F\star_l G) = \lambda(F\star_l G) \\ 
	\rho(F) \circ \rho(G) = (FI) \circ (GI) 
		&= F\circ(GI)\cdot GI = (F\star_r G) I  = \rho(F\star_r G).
\end{align*}
\end{proof}

\begin{rmk}
The groups $(G_B^{I,l}, \circ)$ and $(G_B^{I,r},\circ)$ are naturally in bijection with $G_B^\inv$, via respectively $\lambda$ and $\rho$. The group law $\star_l$ and $\star_r$ are the translation of these groups over the same set $G_B^\inv$, which allows the two groups to interact.
\end{rmk}

Let us now take inspiration from Theorem \ref{thm:S-transf_twisted_factorisation}.

\begin{defn}
\label{def:sqdot1}
For $F,G \in G_B^\inv$, we define the product on $G_B^\inv$
\begin{equation}
\label{eq:squarebox}
	F \sqdot G := G \cdot F \circ (G^{-1} I G).
\end{equation}
\end{defn}

\begin{lem}
\label{lem:sqdot_group_law}
The set $G_B^{\mathrm{inv}}$ forms a group for the product \eqref{eq:squarebox}, with unit $1$.
\end{lem}

\begin{proof}
The unit: For $F\in G_B^\inv$,
\begin{align*}
	F \sqdot 1 &= 1 \cdot F \circ (1 \cdot I \cdot 1) = F \\
	1 \sqdot F &= F \cdot 1 \circ (F^{-1} I F) = F \cdot 1 = F.
\end{align*}
The inverse of $F$ with respect to the product \eqref{eq:squarebox} is given by 
$$
	F^{\sqdot -1}=F^{-1}\circ (F^{-1} I F)^{\circ-1}.
$$ 
Indeed, since
\begin{align*}
	(F^{-1}\circ (F^{-1} I F)^{\circ-1}) \sqdot F &= F \cdot (F^{-1}\circ (F^{-1} I F)^{\circ-1}) \circ (F^{-1} I F) = F\cdot F^{-1} = 1.
\end{align*}
Finally, we show associativity. First, note that for $F,G,H \in G_B^\inv$ we have
\begin{align}
	(G\sqdot H)^{-1} I (G\sqdot H) &= (H \cdot G \circ (H^{-1} I H))^{-1} \cdot I \cdot H \cdot G \circ (H^{-1} I H) \nonumber\\ 
	&= (G \circ (H^{-1} I H))^{-1} \cdot H^{-1} \cdot I \cdot H \cdot G \circ (H^{-1} I H) \nonumber\\ 
	&= G^{-1} \circ (H^{-1} I H) \cdot H^{-1} \cdot I \cdot H \cdot G \circ (H^{-1} I H) \nonumber\\ 
	&= (G^{-1} I G) \circ (H^{-1} I H), \label{eq:F-1IF_of_sqdot}
\end{align}
and therefore
\begin{align*}
	(F\sqdot G)\sqdot H &= (G \cdot F \circ (G^{-1} I G)) \sqdot H \\ 
	&= H \cdot (G \cdot F \circ (G^{-1} I G)) \circ (H^{-1} I H) \\ 
	&= H \cdot G \circ (H^{-1} I H) \cdot F \circ (G^{-1} I G) \circ (H^{-1} I H) \\ 
	&= (G\sqdot H) \cdot F \circ ((G\sqdot H)^{-1} I (G\sqdot H)) \\ 
	&= F\sqdot (G\sqdot H).
\end{align*}
\end{proof}

\begin{rmk}
The map
\[
	(G_B^\inv, \sqdot) \to (G_B^\dif, \circ), \quad F \mapsto F^{-1}IF,
\]
is a group morphism. Note, however, that it is not necessarily injective, so we cannot use the trick of "transporting back the structure". In fact, if $B$ is commutative (in particular in the scalar-valued case, $B=\mathbb{K}$), then this map is constant equal to $I$, the identity element of $(G_B^\dif, \circ)$.
\end{rmk}

\begin{defn}
\label{def:squareboxopp}
For $F,G\in G_B^\mathrm{inv}$, define also the product
\begin{equation}
\label{eq:squareboxopp}
	F \sqdot' G := F \circ (GIG^{-1}) \cdot G.
\end{equation}
\end{defn}

\begin{prop}
The set $G_B^\mathrm{inv}$ forms a group for the product \eqref{eq:squareboxopp}. Moreover, the map $\sigma : (G_B^\mathrm{inv}, \sqdot) \to (G_B^\mathrm{inv}, \sqdot')$, $\sigma(F):=F^{-1}$ is a group isomorphism.
\end{prop}

\begin{proof}
It suffices to verify that for $F,G\in G_B^\mathrm{inv}$, $\sigma(F\sqdot G) = \sigma(F) \sqdot' \sigma(G)$. Indeed,
\begin{align*}
	\sigma(F\sqdot G) 	
		&= (F\sqdot G)^{-1} = (G\cdot F\circ(G^{-1}IG))^{-1} 
		= F^{-1} \circ (G^{-1}IG) \cdot G^{-1} 
		= \sigma(F) \sqdot' \sigma(G).
\end{align*}
\end{proof}


\subsection{Post-groups on \texorpdfstring{$G_B^\inv$}{G B inv}}
\label{ssec:post-groups}

We will now show that several of the group laws described in the previous section fall under the description of post-groups. The notion of post-group appeared recently in \cite{Bai_Guo_Sheng_Tang_post_groups} (see also \cite{mencattini-quesney2021}, and \cite{al-kaabi_ebrahimi-fard_manchon} for more details). Examples of post-groups can be found in (\cite{guin-oudom2008} and) \cite{Lundetal2015}.

\begin{defn} \cite{Bai_Guo_Sheng_Tang_post_groups}
\label{def:postgroup}
A \textit{post-group} $(G,\cdot,\rhd)$ consists of a group $(G,\cdot)$, endowed with a map (or \textit{post-group action}) $\rhd : G\times G \to G$, such that $G \ni a \mapsto L_b^\rhd(a):=b \rhd a \in G$ is an automorphism for all $b\in G$, and such that for any $a,b,c\in G$:
$$
	(a *_\rhd b) \rhd c = a \rhd (b \rhd c),
$$
where the Grossman--Larson product
\begin{equation}
\label{def:GLprod}
	a *_\rhd b := a \cdot (a \rhd b).
\end{equation}
For a post-group $(G,\cdot,\rhd)$, the Grossman--Larson product \eqref{def:GLprod} defines a new group law on $G$, called the \textit{Grossman--Larson} law. If the group $(G,\cdot)$ is commutative, then $(G,\cdot,\rhd)$ is called \textit{pre-group}.
\end{defn}

\begin{defn} \cite{Bai_Guo_Sheng_Tang_post_groups}
\label{def:postgroupmorphism}
A morphism $\Gamma: G \to H$ between post-groups $(G,\cdot_{\scriptscriptstyle{G}},\rhd_{\scriptscriptstyle{G}})$ and $(H,\cdot_{\scriptscriptstyle{H}},\rhd_{\scriptscriptstyle{H}})$ by definition satisfies for all $g_1,g_2 \in G$
$$
	\Gamma(g_1 \cdot_{\scriptscriptstyle{G}} g_2)=\Gamma(g_1) \cdot_{\scriptscriptstyle{H}} \Gamma(g_2)
	\quad \text{and} \quad
	\Gamma(g_1 \rhd_{\scriptscriptstyle{G}} g_2)=\Gamma(g_1) \rhd_{\scriptscriptstyle{H}} \Gamma(g_2).
$$
\end{defn}

Let us now define four post-group actions.

\begin{defn}
For $F,G\in G_B^\inv$, let
\begin{equation}
\label{def:actions}
\begin{matrix}
	F \rhd_l G :=& \!\! \!\! \!\! G \circ(IF) \\[0.2cm] 
	F \rhd_r G :=& \!\! \!\! \!\! G \circ(FI) 
\end{matrix}
\qquad\ \text{and} \qquad\
\begin{matrix}
	F \rhd G :=& \!\! \!\! \!\! G \circ(F^{-1}IF) \\[0.2cm]
	F \rhd' G :=& \!\! \!\! \!\! G \circ(FIF^{-1})
\end{matrix}.
\end{equation}
\end{defn}

For a product $\bullet$ on $G_B^\mathrm{inv}$, we denote its opposite product by $\overline\bullet=\bullet \circ \tau$, where $\tau(a,b):=(b,a)$. 

\begin{lem}
The group $(G_B^\mathrm{inv},\cdot)$ can be equipped with the following post-group structures:
\begin{align*}
	(G_B^\mathrm{inv}, \cdot, \rhd_l) 
		& \text{ with } \overline\star_l \text{ being its Grossman--Larson law,} \\ 
	(G_B^\mathrm{inv}, \overline\cdot, \rhd_r) 
		& \text{ with } \overline\star_r \text{ being its Grossman--Larson law,}
\end{align*}
and 
\begin{align*}
	(G_B^\mathrm{inv}, \cdot, \rhd) 
		& \text{ with } \overline\sqdot \text{ being its Grossman--Larson law,} \\ 
	(G_B^\mathrm{inv}, \overline\cdot, \rhd') 
		& \text{ with } \overline\sqdot' \text{ being its Grossman--Larson law.}
\end{align*}
\end{lem}

\begin{proof}
For $F \in G_B^\inv$, we have that $IF$, $FI$, $F^{-1}IF$, and $FIF^{-1}$ are in $G_B^\dif$, so all maps $F\rhd_l -$, $F\rhd_r -$, $F\rhd -$, $F\rhd' -$ are invertible. Moreover, they are group morphisms thanks to Lemma \ref{lem:rightaction}, so they are group automorphisms. Now, let us fix $F,G,H \in G_B^\inv$. We have 
\begin{align*}
	F \rhd_l (G\rhd_l H) 
	= H \circ (IG) \circ (IF) 
	= H \circ ((IF)\cdot G\circ(IF)) 
	= H \circ(I \cdot(F \cdot (F\rhd_l G))) 
	= (F\cdot (F\rhd_l G))\rhd_l H,
\end{align*}
so that $(G_B^\inv, \cdot, \rhd_l)$ is a post-group with Grossman--Larson product 
$$
	F*_{\rhd_l} G 
	= F \cdot (F\rhd_l G) 
	= F\cdot G\circ(IF) 
	= G \star_l F 
	= F \overline \star_l G. 
$$
The same goes for the action $\rhd_r$.

Regarding the action $\rhd$, the argument is quite similar. We have
\begin{align*}
	F \rhd (G \rhd H) 
	&= H \circ(G^{-1}IG)\circ(F^{-1}IF) \\ 
	&= H\circ(G^{-1}\circ(F^{-1}IF)\cdot F^{-1}IF \cdot G\circ (F^{-1}IF)) \\ 
	&= H \circ \Big((F \cdot G\circ (F^{-1}IF))^{-1} I (F \cdot G\circ (F^{-1}IF))\Big) \\ 
	&= H \circ \Big((F \cdot F\rhd G)^{-1} I (F \cdot F\rhd G)\Big) \\ 
	&= (F\cdot (F\rhd G)) \rhd H,
\end{align*}
and 
$$
	F\cdot (F\rhd G) = F \cdot G\circ (F^{-1}IF) = G \sqdot F = F ~\overline \sqdot ~G. 
$$
The same goes for the action $\rhd'$.
\end{proof}

\begin{rmk}
\label{rmk:pre-groups}
Recall that a pre-group is a post-group with commutative group product. Therefore, if $B$ is commutative, then  $(G_B^\inv, \cdot)$ is a commutative group and the left- and right actions coincide, that is, $\gg:=\rhd_l=\rhd_r$, which makes  $(G_B^\inv, \cdot, \gg)$ a pre-group with Grossman--Larson product $\overline{\star}:=\overline{\star}_l= \overline{\star}_r$.
\end{rmk}

\begin{rmk}
\label{rmk:post_group_corresp_inverse_group}
Let $(G, \cdot)$ be any group. There is a bijective correspondence between post-group actions over $(G, \cdot)$ and over $(G, \overline \cdot)$. Indeed, if $\rhd_1$ is a post-group action over $(G, \cdot)$, then $\rhd_1'$ defined by $x \rhd_1' y := x^{-1} \rhd_1 y$ is a post-group action over $(G, \overline \cdot)$. In this way, the inverse map is an isomorphism of post-groups between $(G, \cdot, \rhd_1)$ and $(G, \overline\cdot, \rhd_1')$.

\medskip

Following Remark \ref{rmk:post_group_corresp_inverse_group}, we get two more post-groups, $(G, \overline \cdot, \rhd_l')$ and $(G, \cdot, \rhd_r')$, with actions 
\begin{align*}
	F \rhd_l' G &:= G \circ(IF^{-1}) \\ 
	F \rhd_r' G &:= G \circ(F^{-1}I).
\end{align*}
Moreover, the actions $\rhd$ and $\rhd'$ are linked through this correspondence. 
\end{rmk}

The formalism of post-groups provides descriptions for the inverse maps for the Grossman--Larson products $\star_l$, $\star_r$, $\sqdot\,$, and $\sqdot'$.

\begin{prop}
\label{prop:inverse_maps_S}
The following maps $G_B^\mathrm{inv} \to G_B^\mathrm{inv}$ define inverse maps for the specified group laws on $G_B^\mathrm{inv}$.
\[
\begin{array}{rcccl}
	\sigma :		& F & \mapsto & F^{-1} 					& \text{ is the inverse map for } \cdot \\
	S^l :			& F & \mapsto & F^{-1}\circ(IF)^{\circ-1} 		& \text{ is the inverse map for } \star_l \\
	S^r :			& F & \mapsto & F^{-1}\circ(FI)^{\circ-1} 		& \text{ is the inverse map for } \star_r \\
	S^{\sqdot} :	& F & \mapsto & F^{-1}\circ(F^{-1}IF)^{\circ-1} 	& \text{ is the inverse map for } \sqdot \\
	S^{\sqdot'} :	& F & \mapsto & F^{-1}\circ(FIF^{-1})^{\circ-1} 	& \text{ is the inverse map for } \sqdot'.
\end{array}
\]
In particular, all these maps are involutive, i.e.~$S(S(F)) = F$ for $S \in \{\sigma, S^l,S^r, S^{\sqdot}, S^{\sqdot'}\}$ and $F\in G_B^\mathrm{inv}$.
\end{prop}

\begin{proof}
In a post-group $(G, \cdot, \rhd)$, the map $y \mapsto L^{\rhd}_x(y)= x\rhd y$ is an automorphism. Let us denote its inverse by $y \mapsto x \rhd^{-1} y$. Then the inverse of $x\in G$ for $*_\rhd$ is given by $x \rhd^{-1} x^{-1}$, where $x^{-1}$ is the inverse of $x$ in $(G, \cdot)$.
\end{proof}

\begin{rmk}
Going back to Lemma \ref{lem:fixed_point_S_f}, we note that \eqref{eq:fixed_point_S_f} can be written
$$
\begin{array}{ll}
	S^l(F)&=S^l(F) \rhd_l \sigma(F)\\[0.2cm]
	S^r(F)&=S^r(F) \rhd_r \sigma(F)
\end{array}
\qquad \text{and} \qquad
\begin{array}{ll}
	S^{\sqdot}(F)&=S^{\sqdot}(F) \rhd \sigma(F)\\[0.2cm]
	S^{\sqdot'}(F)&=S^{\sqdot'}(F) \rhd' \sigma(F) 
\end{array},
$$
relating $S^l$, $S^r$, $S^{\sqdot}$, $S^{\sqdot'}$, and $\sigma$ via the corresponding post-Lie actions. 

Note that this way of defining $S^l$ and $S^r$ coincides with Definition \ref{def:S_transform} of the respective S-transforms. Also note that the S-transform is now considered a map between multilinear function series, instead of a particular series $S_a = S^l(K_a)$. This allows us to apply these maps more generally than just to the R-transform, i.e., series of free cumulants. As such they will be used extensively in the rest of the section. 

We note that this viewpoint is in line with the map $\mathcal{F}$ defined in Nica and Speicher \cite[p.293]{nica_speicher_book} in the scalar-valued setting. Further below, we state the analog of Theorem 18.14 \cite[p.294]{nica_speicher_book} identifying the S-transform as a group isomorphism in the operator-valued setting. 
\end{rmk}

Let us also point at a connection between the Grossman--Larson laws $\star_l$, $\star_r$, $\sqdot\,$, and $\sqdot'$.

\begin{prop}
\label{prop:rhd_from_rhd_l_and_rhd_r}
For $F,G\in G_B^\mathrm{inv}$, we have
\begin{align*}
	F \rhd G &= F \rhd_l (S^l(F) \rhd_r G), \\ 
	F \rhd' G &= F \rhd_r (S^r(F) \rhd_l G).
\end{align*}
Therefore
\begin{align*}
	F \sqdot G &= (S^l(G)\rhd_r F) \star_l G, \\ 
	F \sqdot' G &= (S^r(G)\rhd_l F) \star_r G.
\end{align*}
\end{prop}

\begin{proof}
We have
\[
	F \rhd_l (S^l(F) \rhd_r G) 
	= G \circ (S^l(F)I) \circ (IF) 
	= G \circ (S^l(F)\circ(IF) IF) 
	= G \circ (F^{-1}IF) 
	= F \rhd G,
\]
and similarly for $F\rhd' G$. Thus 
\[
	F \sqdot G 
	= G \cdot (G \rhd F) 
	= G \cdot G \rhd_l (S^l(G) \rhd_r F) 
	= (S^l(G)\rhd_r F) \star_l G,
\]
and similarly for $F \sqdot' G := F \circ (GIG^{-1}) \cdot G$ introduced in Definition \ref{def:squareboxopp}.
\end{proof}

Section \ref{ssec:more_relations} investigates further this unexpected connection.


\subsection{The boxed convolution as a group operation}
\label{ssec:boxconv_is_group_law}

Using the (Grossman--Larson) product $\sqdot$ permits to express the identity \eqref{eq:S-transf_twisted_factorisation} in Theorem \ref{thm:S-transf_twisted_factorisation} more compactly
\begin{equation}
\label{eq:twisted_factorisation}
	S^l(K_{ab}) = S^l(K_a) \sqdot S^l(K_b).
\end{equation}
We would like to write directly
\begin{equation}
\label{eq:boxcon_prop}
	K_{ab} = K_a \boxcon K_b
\end{equation}
for some product $\boxcon$, that will be called the \textit{boxed convolution}. This would generalise the boxed convolution opertion from Nica and Speicher \cite{nica_speicher_book} to the operator-valued setting.

\begin{defn}
\label{def:boxconvol}
For $F,G\in G_B^\mathrm{inv}$, define 
\begin{equation}
\label{eq:boxconvol}
	F \boxcon G := S^l(S^l(F) \sqdot S^l(G))
\end{equation}
\end{defn}

From this definition and Theorem \ref{thm:S-transf_twisted_factorisation}, it is clear that equation \eqref{eq:boxcon_prop} holds. Let us show that the boxed convolution operation carries more structure than just satisfying this equation. First, it is not dependant on a left-right choice, as opposed to Theorem \ref{thm:S-transf_twisted_factorisation}. Indeed, we have the following identity:

\begin{lem}
For $F,G\in G_B^\mathrm{inv}$,
\begin{equation}
\label{eq:inverserelation}
	 S^r(S^r(G)\sqdot' S^r(F)) = S^l(S^l(F)\sqdot S^l(G)) = F \boxcon G.
\end{equation}
\end{lem}

\begin{proof}
We will see in Section \ref{ssec:more_relations} that $S^{\sqdot} = S^lS^r\sigma = \sigma S^rS^l$. Therefore
\begin{align*}
	S^l(S^l(F)\sqdot S^l(G)) &= S^r\sigma S^{\sqdot}(S^{\sqdot}\sigma S^r(F) \sqdot S^{\sqdot}\sigma S^r(G)) \\ 
	&= S^r \sigma (\sigma S^r(G) \sqdot \sigma S^r(F)) \\ 
	&= S^r ( S^r(G) \sqdot' S^r(F)). 
\end{align*}
\end{proof}

\begin{rmk}
\label{rmk:def_boxconv}
\begin{enumerate}
	\item The set $G_B^\inv$ forms a group for the boxed convolution product \eqref{eq:boxconvol}, and 
	\[
		\begin{array}{cccc}
		S^l : 	& (G_B^\inv, \sqdot) 	& \to 	& (G_B^\inv, \boxcon) \\[0.2cm]
		S^r : 	& (G_B^\inv, \sqdot') 	& \to 	& (G_B^\inv, \overline{\boxcon})
		\end{array}	
	\]
	are isomorphisms of groups.
	
	\item The product $\boxconv$ defined over $G_B^{I,l}$ which we introduced in \cite[Definition 3.16]{part1} satisfies
	\[
		(IK_{ab}) = (IK_a) \boxconv (IK_b).
	\]
	In fact, since the series $K_a$ and $K_b$ can be replaced by any other series in $G_B^\inv$, the product $\boxconv$ relates to the product $\boxcon$ in \eqref{eq:boxconvol}, as follows: for $F,G\in G_B^\inv$, 
	\[
		I(F\boxcon G) = (IF) \boxconv (IG).
	\]
	Moreover, the map $\lambda : (G_B^\inv, \boxcon) \to (G_B^{I,l},\boxconv)$, $F \mapsto IF$, is an isomorphism of groups.
\end{enumerate}
\end{rmk}

Surprisingly, the inverse map for the product \eqref{eq:boxconvol} is quite simple.

\begin{prop}
The inverse map for the product $\boxcon$ defined in \eqref{eq:boxconvol} is given, for $F\in G_B^\mathrm{inv}$, by
\[
	S^{\tiny\smboxcon}(F) = F^{-1}\circ(FIF)^{\circ-1}.
\]
\end{prop}

\begin{proof}
We have
\begin{align*}
	S^{\tiny\smboxcon}(F) &= S^l(S^{\sqdot}(S^l(F))) \\ 
	&= S^l(S^{\sqdot}(F^{-1}\circ(IF)^{\circ-1})) \\ 
	&= S^l\Big( F\circ(IF)^{\circ-1} \circ (F\circ(IF)^{\circ-1} I F^{-1}\circ(IF)^{\circ-1})^{\circ-1}\Big) \\ 
	&= S^l\Big( F \circ (IF)^{\circ-1} \circ (IF) \circ(F(IF)F^{-1})^{\circ-1} \Big) \\ 
	&= S^l(F \circ (FI)^{\circ-1}) \\ 
	&= F^{-1} \circ (FI)^{\circ-1} \circ (IF \circ (FI))^{\circ-1} \\ 
	&= F^{-1} \circ (FI)^{\circ-1} \circ (FI) \circ ((FI)F)^{\circ-1} \\ 
	&= F^{-1}\circ(FIF)^{\circ-1}.
\end{align*}
\end{proof}

Moreover, two other expressions for the product $F\boxcon G$ using the post-groups $(G_B^\inv, \cdot, \rhd_l)$ and $(G_B^\inv, \overline\cdot, \rhd_r) $ can also be derived.

\begin{prop}
\label{prop:eqs_boxcon_subord}
For $F,G\in G_B^\mathrm{inv}$, we have
\begin{align*}
	F \boxcon G &= G \star_l S^l(G \rhd_r S^l(F)) \\ 
	F \boxcon G &= F \star_r S^r(F \rhd_l S^r(G)).
\end{align*}
\end{prop}

\begin{proof}
We know from Proposition \ref{prop:rhd_from_rhd_l_and_rhd_r} that $F \sqdot G = (S^l(G)\rhd_r F) \star_l G$. Therefore 
\begin{align*}
	F \boxcon G &= S^l(S^l(F) \sqdot S^l(G)) \\ 
	&= S^l((G\rhd_r S^l(F)) \star_l S^l(G)) \\ 
	&= G \star_l S^l(G \rhd_r S^l(F)).
\end{align*}
The second equality follows similarly from $F \sqdot' G = (S^r(G)\rhd_l F) \star_r G$.
\end{proof}

\begin{rmk}
\label{rmk:commutative}
In the case of a commutative $B$, i.e., when $ \star_l =  \star_r$ and $ \rhd_r = \rhd_l$, Proposition \ref{prop:eqs_boxcon_subord} is consistent with the commutativity of the boxed convolution in  \cite{nica_speicher_book} 
$$
	F \boxcon G
	= G \boxcon F.
$$
\end{rmk}

We can also derive alternative expressions of the operator-valued moment-cumulant relations.

\begin{lem}
\label{lem:equiv_MC_relations}
Let $M_a, K_a \in G_B^\mathrm{inv}$. The following statements are equivalent.
\begin{align}
	IK_a &= (IM_a) \circ(I + IM_aI)^{\circ-1} \label{eq:MC_1}\\
	K_aI &= (M_aI) \circ(I + IM_aI)^{\circ-1} \label{eq:MC_2}\\
	I + IM_aI &= ((1+IK_a)^{-1}I)^{\circ-1} \label{eq:MC_3} \\ 
		   &= (I(1+K_aI)^{-1})^{\circ-1} \label{eq:MC_4} \\ 
	1+IM_a &= S^r((1+IK_a)^{-1}) \label{eq:MC_5} \\ 
	1+M_aI &= S^l((1+K_aI)^{-1}) \label{eq:MC_6}
\end{align}
These are free moment-cumulant relations: for any random variable $a\in \mathcal A$, the series $M_a$ and $K_a$ satisfy all these conditions.
\end{lem}

\begin{proof}
Equivalence \eqref{eq:MC_3} $\iff$ \eqref{eq:MC_4} is clear since 
$$
	(1+IK_a)^{-1}I 
	= I - IK_aI + IK_aIK_aI - IK_aIK_aIK_aI +\cdots 
	= I(1+K_aI)^{-1}.
$$

Equivalences \eqref{eq:MC_3} $\iff$ \eqref{eq:MC_5} and \eqref{eq:MC_4} $\iff$ \eqref{eq:MC_6} come from the facts that 
$$
	((1+IK_a)^{-1}I)^{\circ-1} = S^r((1+IK_a)^{-1})I
$$ 
respectively  
$$
	(I(1+K_aI)^{-1})^{\circ-1} = IS^l((1+K_aI)^{-1}).
$$
Moreover, equation \eqref{eq:MC_5} can be written as 
$$
	1+IK_a 
	= S^r(1+IM_a)^{-1} 
	= (1+IM_a) \circ(I + IM_aI)^{\circ-1} 
	= 1 + (IM_a) \circ(I + IM_aI)^{\circ-1},
$$ 
which shows \eqref{eq:MC_5} $\iff$ \eqref{eq:MC_1}. Similarly, we have \eqref{eq:MC_6} $\iff$ \eqref{eq:MC_2}.
\end{proof}

\begin{defn}
Let us define the zeta series as 
\[
	\zeta := 1 + I + I^2 +\cdots = (1-I)^{-1}.
\]
\end{defn}

\begin{lem}
\label{lem:S^l_zeta}
For $F\in G_B^\mathrm{inv}$, we have
\begin{align*}
	S^l(\zeta) = S^r(\zeta) = (-1) \rhd_r \zeta = (-1) \rhd_l \zeta = (1+I)^{-1} = 1 - I + I^2 - I^3 +\cdots.
\end{align*}
\end{lem}

\begin{proof}
Indeed, going back to  Lemma \ref{lem:fixed_point_S_f}, we see that using the first fixed point equation in \eqref{eq:fixed_point_S_f} gives 
\[
	S^l(\zeta) 
		= S^l(\zeta) \rhd_l \sigma(\zeta) 
		=  S^l(\zeta) \rhd_l \zeta^{-1}  
		= S^l(\zeta) \rhd_l (1-I) 
		= 1 - I S^l(\zeta).
\]
And therefore $(1+I) S^l(\zeta) = 1$ and $S^l(\zeta) = (1+I)^{-1} = 1 - I + I^2 - I^3 +\cdots = \zeta \circ(-I)$. The same argument goes for $S^r(\zeta)$.
\end{proof}

\begin{prop}
\label{prop:zetaboxconvol}
Let $M_a,K_a\in G_B^\mathrm{inv}$ be respectively the moment series and the cumulant series of a random variable $a  \in \mathcal{A}$, i.e., two series such that the equalities in Lemma \ref{lem:equiv_MC_relations} hold. Then
\begin{equation}
\label{eq:boxconvolzeta}
	M_a = K_a \boxcon \zeta.
\end{equation}
\end{prop}

\begin{proof}
From Proposition \ref{prop:eqs_boxcon_subord} we get
\begin{align*}
	K_a \boxcon \zeta &= K_a \star_l S^l(K_a \rhd_r S^l(\zeta)) \\ 
	&= K_a \star_l S^l(K_a \rhd_r (1+I)^{-1}) \\ 
	&= K_a \star_l S^l((K_a \rhd_r (1+I))^{-1}) \\ 
	&= K_a \star_l S^l((1+K_aI)^{-1}) \\ 
	&= K_a \star_l (1+M_aI) \\ 
	&= M_a.
\end{align*}
Note that the second equality follows from Lemma \ref{lem:S^l_zeta}, the fifth from equation \eqref{eq:MC_6}, and the last one from equation \eqref{eq:MC_1}. Indeed, $IM_a = (IK_a)\circ(I+IM_aI) = I(K_a \star_l (1+M_aI))$.
\\
We note that a second proof of Proposition \ref{prop:zetaboxconvol} can be given using Proposition \ref{prop:S_l_r_sqdot_related} and equation \eqref{eq:Sl_K_Sl_M_sqdot} (see Remark \ref{rmk:proof2} below). 
\end{proof}

\begin{rmk}
\label{rmk:NicaSpeicherzeta}
Observe that the moment-cumulant relation \eqref{eq:boxconvolzeta} in Proposition \ref{prop:zetaboxconvol} can be interpreted as the operator-valued generalisation of \cite[Proposition 17.4, p.274]{nica_speicher_book}.
\end{rmk}


\subsection{Subordination and other boxed convolutions}
\label{ssec:subordination_other_boxed_conv}

We provide initial motivations for this section. However, readers unfamiliar with monotone or boolean independence may skip directly to Definition \ref{def:H1H2}. We will primarily concentrate on the scalar-valued setting. In \cite{Lencz2007,Lencz2008} and \cite{Nica2009}, the notion of subordination \cite{biane1998,voiculescu1993} is used  to express (commutative) free and boolean additive convolutions in terms of monotone additive convolution. Recall that the noncommutative distribution $\mu_a$ of a noncommutative random variable $a \in \mathcal A$ can be represented as a moment series
\[
	M_a(z) = \sum_{n\ge 0} \bb E(a^{n+1})z^n = \bb E(a) + \bb E(a^2)z + \bb E(a^3)z^2 +\cdots.
\]
We define its $F$-transform to be 
\begin{align*}
	F_a(z) = z F_a'(z) &:= z \cdot \left(1 + \frac 1 z\cdot M_a\left(\frac 1 z\right)\right)^{-1} \\ 
	& = z - \bb E(a) + \big(\bb E(a)^2 - \bb E(a^2)\big)\frac 1 z + \big(\bb E(a^3) - 2 \bb E(a^2)\bb E(a) + \bb E(a)^3\big) \frac 1 {z^2}+ \cdots.
\end{align*}
Define also the $\eta$-transform
\begin{align*}
	\eta_a(z) = z\eta'_a(z) 
	& := zM_a(z) \cdot (1 + z\cdot M_a(z))^{-1} \\
	& = zM_a(z) F_a'\left(\frac 1 z\right)\\ 
	&= z \Big( 
	\bb E(a) 
	+ \big(\bb E(a^2) - \bb E(a)^2\big)z 
	+ \big(\bb E(a^3) - 2 \bb E(a^2)\bb E(a) + \bb E(a)^3\big) z^2 \\
	&\qquad\  + \big(\bb E(a^4) + 3 \bb E(a)^2 \bb E(a^2) - 2 \bb E(a) \bb E(a^3) - \bb E(a)^4\big)z^3 + \cdots
	\Big).
\end{align*}
As in the rest of the paper, we focus on reduced series, i.e., expansions starting with a constant term. That is why we consider primed objects, $F'_a$ and $\eta'_a$, instead of $F_a$ and $\eta_a$. The two series are in fact very similar as we have:
\[
	F_a'\left(\frac 1 z\right) 
	= \dfrac 1 {1 + z\cdot M_a(z)} 
	=  1 - \dfrac {z\cdot M_a(z)} {1 + z\cdot M_a(z)} 
	= 1 - z \cdot \eta'_a(z).
\]
We can define similarly $F'_\mu$ and $\eta'_\mu$ for any noncommutative distribution $\mu$. Following Lenczewski \cite{Lencz2007}, the monotone and boolean additive convolutions of two distributions $\mu$ and $\nu$ are expressed through the $F$-transform
\begin{align}
	F_{\mu\,  *\,  \nu}(z) 	 &= F_\mu (F_\nu(z)) \\
	F_{\mu\,  \uplus\,  \nu}(z) &= F_\mu(z) + F_\nu(z) - z,
\end{align}
where $*$ denotes monotone additive convolution (Lenczewski writes $\rhd$), and $\uplus$ is boolean additive convolution. Note that this can be reformulated as
\begin{align}
	F'_{\mu\,  *\,  \nu}(z) 
		&= (F'_\mu \star F'_\nu)(z) 
		  := F'_\nu(z) \cdot F'_\mu(zF'_\nu(z)) \\
	\eta'_{\mu\,  \uplus\,  \nu}(z) 
		&= \eta'_\mu(z) + \eta'_\nu(z).
\end{align}
Now, following Lenczewski \cite{Lencz2008}, the same holds for boolean and monotone multiplicative convolutions, this time with the help of the $\eta$-transform
\begin{align}
	\eta_{\mu\,  \circlearrowright\,  \nu}(z) 
		&= \eta_\mu (\eta_\nu(z)) \\ 
	\eta'_{\mu\,  \smstarbox\,  \nu}(z) 
		&= \eta'_\mu(z)\eta'_\nu(z),
\end{align}
where $\circlearrowright$ denotes monotone multiplicative convolution and $\starbox$ stands for boolean multiplicative convolution (not to be confused with the boxed convolution $\!\boxcon$). The first equation can again be reformulated using the star product $\star$
\begin{equation}
	\eta'_{\mu\,  \circlearrowright\,  \nu}(z) 
		= (\eta'_\mu \star \eta'_\nu)(z) = \eta'_\nu(z) \cdot \eta'_\mu(z\eta'_\nu(z)).
\end{equation}
\\
Moreover, moving to free additive and free multiplicative convolution, there exist so-called \textit{subordination} distributions, $\mu\boxsemiplus \nu$ and $\mu \boxsemitimes \nu$, such that we have the following decompositions
\begin{align}
	\mu \boxplus \nu 	&= \mu * (\mu \boxsemiplus \nu) = \nu * (\nu \boxsemiplus \mu) \\
	\mu \boxplus \nu 	&= (\mu \boxsemiplus \nu) \uplus (\nu \boxsemiplus \mu)  \\
	\mu \boxtimes \nu 	&= \mu \circlearrowright (\nu \boxsemitimes \mu) 
					   = \nu \circlearrowright (\mu \boxsemitimes \nu) \\ 
	\mu \boxtimes \nu 	&= (\mu \boxsemitimes \nu) \starbox (\nu \boxsemitimes \mu),
\end{align}
where $\boxplus$ denotes free additive convolution and $\boxtimes$ stands for free multiplicative convolution. In terms of $F'$-transforms and $\eta'$-transforms, this amounts to
\begin{align}
	F'_{\mu\,  \boxplus\,  \nu} 
		&= F'_\mu \star F'_{\mu\, \smboxsemiplus\, \nu} 
		  = F'_\nu \star F'_{\nu\,  \smboxsemiplus\,  \mu} \\ 
	F'_{\mu\,  \boxplus\,  \nu} 
		&= F'_{\nu\,  \smboxsemiplus\,  \mu} + F'_{\mu\,  \smboxsemiplus\,  \nu} - 1 \\ 
	\eta'_{\mu\,  \boxtimes\,  \nu} 
		&= \eta'_\mu \star \eta'_{\mu\,  \smboxsemitimes\,  \nu} 
		  = \eta'_\nu \star \eta'_{\nu\,  \smboxsemitimes\,  \mu} \label{eq:boxtimes_monot}\\ 
	\eta'_{\mu\,  \boxtimes\,  \nu} 
		&= \eta'_{\nu\,  \smboxsemitimes\,  \mu} \cdot \eta'_{\mu\,  \smboxsemitimes\,  \nu}. \label{eq:boxtimes_bool}
\end{align}

Knowing that the boxed convolution $\!\boxcon$ expresses free multiplicative convolution when applied to cumulant series, we can compare Proposition \ref{prop:eqs_boxcon_subord} with Equations \eqref{eq:boxtimes_monot}. Although we are dealing with cumulant series, or $R$-transforms, and not $\eta$-transforms, it still seems natural to deduce two series that have interesting properties and that can be thought of as operator-valued multiplicative subordination. In fact, we will see that, for example, Equation \eqref{eq:boxtimes_bool} still holds in our setting.

\begin{defn}
\label{def:H1H2}
For $F,G\in G_B^\inv$, define
\begin{align}
\label{eq:def_H1H2}
\begin{aligned}
	H_1(F,G) &:= S^l(G\rhd_r S^l(F)) \\ 
	H_2(F,G) &:= S^r(F\rhd_l S^r(G)).
\end{aligned}
\end{align}
\end{defn}

The interpretation of the following identities \eqref{eq:H1H2} as operator-valued multiplicative subordination becomes compelling in light of the fact that the two products, $\star_l $ and $\star_r $, (see \eqref{eq:starproducts} in Definition \ref{def:star_l_star_r}) can be interpreted as two ways of defining operator-valued additive monotone convolution (see \cite[Definition 4.4]{HasebeSaigo2014}).

\begin{lem}
\label{lem:boxcon_H_1_H_2}
Let $F,G\in G_B^\mathrm{inv}$ and $H_1 = H_1(F,G)$, $H_2 = H_2(F,G)$ as in Definition \ref{def:H1H2}. Then 
\begin{align}
\label{eq:H1H2}
\begin{aligned}
	F \boxcon G &= G \star_l H_1 \\ 
	F \boxcon G &= F \star_r H_2 \\ 
	F \boxcon G &= H_1 \cdot H_2, 
\end{aligned}
{}\end{align}
and 
\begin{align}
	H_1 = H_2 \rhd_r F \label{eq:H1_is_H2_rhdr_F}\\ 
	H_2 = H_1 \rhd_l G. \label{eq:H2_is_H1_rhdl_G}
\end{align}
Moreover, 
\begin{align}
\label{eq:GcircH1_is_FcircH2}
	(GI)\circ(IH_1) = (IF)\circ(H_2I) ,
\end{align}
which can be reformulated as follows: for all $L\in G_B^\mathrm{inv}$,
\[
	H_1 \rhd_l (G \rhd_r L) = H_2 \rhd_r (F \rhd_l L).
\]
\end{lem}

\begin{proof}
The first two equalities of \eqref{eq:H1H2} are just a restatement of Proposition \ref{prop:eqs_boxcon_subord}. Here is, however, how we can derive \eqref{eq:H1_is_H2_rhdr_F}:
\begin{align*}
	H_2 \rhd_r F &= S^r(F\rhd_l S^r(G)) \rhd_r F \\ 
	&= F \circ ((F\rhd_l S^r(G))I)^{\circ-1} \\ 
	&= F \circ (S^r(G)\circ(IF) I)^{\circ-1} \\ 
	&= F \circ (IF)^{\circ-1} \circ (S^r(G) (IF)^{\circ-1})^{\circ-1} \\ 
	&= S^l(F) \circ (S^r(G)IS^l(F))^{\circ-1} \\ 
	&= S^l(F) \circ ((GI)^{\circ-1}S^l(F))^{\circ-1} \\ 
	&= S^l(F) \circ (GI) \circ (I S^l(F)\circ(GI))^{\circ-1} \\ 
	&= S^l(S^l(F) \circ (GI)) \\ 
	&= S^l(G\rhd_r S^l(F)) \\ 
	&= H_1.
\end{align*}
Equation \eqref{eq:H2_is_H1_rhdl_G} is derived similarly. Then, the third equality of \eqref{eq:H1H2} 
\[
	F \boxcon G = G \star_l H_1 = H_1 \cdot (H_1 \rhd_l G) = H_1 \cdot H_2.
\]	
Finally, we have
\[
	(GI)\circ(IH_1) = G \circ(IH_1)I H_1 = H_2IH_1 = H_2IF\circ(H_2I) = (IF)\circ(H_2I).
\]
\end{proof}

\begin{rmk}
\label{rmk:connection_part_I}
From this purely algebraic constructions of the subordination series $H_1(F,G)$ and $H_2(F,G)$, one can find back the same series that were constructed from a combinatorial point of view in \cite{part1}. In this paper, four boxed convolution products are defined -- combinatorially, by using planar binary trees at the level of the $n$-linear components --  for $f,g \in G_B^{I,l}$: 
\begin{equation}
\label{eq:four_products_part1}
	f \boxconv g \in G_B^{I,l},
	\qquad f \boxconvline g \in G_B^\dif,
	\qquad f \boxconvred g \in G_B^{I,l},
	\qquad f\boxconvredred g \in G_B^\inv.
\end{equation}
They are shown to satisfy the four relations (\cite[Lemma 3.18 and Lemma 3.19]{part1})
\begin{align}
\label{boxconvolrel}
\begin{array}{ll}
	f \boxconv g =&\!\!\!\!\! g \circ (f \boxconvred g) \\[0.2cm]
	f \boxconv g =&\!\!\!\!\! (f \boxconvred g) \cdot (f \boxconvredred g)
\end{array}
\quad \text{and} \quad 
\begin{array}{ll}
	g \boxconvline f =&\!\!\!\!\! f \circ ((f\boxconvredred g)I) \\[0.2cm]
	g \boxconvline f =&\!\!\!\!\! (f \boxconvredred g) \cdot (f \boxconvred g)
\end{array}.
\end{align}
Let us assume that $F,G\in G_B^\inv$ are such that $f=IF$ and $g=IG$. We claim that the foregoing four products \eqref{eq:four_products_part1} and the four equalities in  \eqref{boxconvolrel} can be described via $H_1 = H_1(F,G)$, $H_2 = H_2(F,G)$, and Lemma \ref{lem:boxcon_H_1_H_2}. More precisely, we claim that
\begin{align*}
f \boxconv g &= I H_1 H_2 = I(F\boxcon G) & f \boxconvred g &= I H_1 \\ 
f \boxconvline g &= H_2 I H_1 & f \boxconvredred g &= H_2,
\end{align*}
and the four equalities \eqref{boxconvolrel} can now be read as
\begin{align*}
\begin{array}{ll}
	I(F\boxcon G) =&\!\!\!\!\! (IG)\circ(IH_1) \\[0.2cm]
	I(F\boxcon G) =&\!\!\!\!\! (IH_1)\cdot H_2
\end{array}
\quad \text{and} \quad 
\begin{array}{ll}
	H_2IH_1 =&\!\!\!\!\! (IF) \circ (H_2I) \\[0.2cm]
	H_2IH_1 =&\!\!\!\!\! H_2\cdot(IH_1)
\end{array}.
\end{align*}
Note that some of them become trivial. To see why the claim is true, recall from Remark \ref{rmk:def_boxconv} that 
\[
	f \boxconv g = I(F \boxcon G).
\]
Writing $f\boxconvred g = IH_i$ and $f\boxconvredred g = H_{ii}$ with $H_i, H_{ii} \in G_B^\inv$, the two equalities on the left of \eqref{boxconvolrel} can be rewritten as
\begin{align*}
	F\boxcon G &= G \star_l H_i \\
	F\boxcon G &= H_i \cdot H_{ii}.
\end{align*}
Comparing these with Lemma \ref{lem:boxcon_H_1_H_2}, we deduce that $H_i = H_1(F,G)$ and $H_{ii} = H_2(F,G)$, and the rest follows.

We therefore have a correspondence between several rather differently defined convolution-type products. On the one hand, the product series \eqref{eq:four_products_part1} are defined "locally", i.e., at the level of their $n$-linear components using combinatorial tools for planar binary trees. On the other hand, the series $F\boxcon G$, $H_1$, $H_2$ are defined at a "global" level, that is, very algebraically, from the set $\Mult[[B]]$ equipped with products $\cdot$ and $\circ$.
\end{rmk}

\begin{rmk}
\label{rmk:newformulas}
\begin{enumerate}

\item[i)]
From \eqref{eq:H1_is_H2_rhdr_F} and \eqref{eq:H2_is_H1_rhdl_G}, we can deduce expressions of $H_1$ and $H_2$ as infinite towers:
\begin{align*}
	H_1 &= (((\cdots \rhd_l G) \rhd_r F) \rhd_l G) \rhd_r F \\ 
	H_2 &= (((\cdots \rhd_r F) \rhd_l G) \rhd_r F) \rhd_l G.
\end{align*}
Note that these towers converge because only the last $n$ terms ($F$ or $G$) influence the term of degree $n$ of the tower.

\item[ii)]
The third identity in \eqref{eq:H1H2}
$$
	F \boxcon G = H_1(F,G) \cdot H_2(F,G)
$$ 
can be seen as the operator-valued analog of \cite[Identity 18.30, p.298]{nica_speicher_book}. We note that one could also consider the identity $f\boxconv g = (f\boxconvred g) \cdot (f \boxconvredred g)$ as the operator-valued analog, but Remark \ref{rmk:connection_part_I} explains equivalence between these identities.

\item[iii)]
Going back to Remark \ref{rmk:commutative}, we can deduce from \eqref{eq:H1H2} the following computational expressions for random variables $a,b \in \mathcal{A}$ with scalar-valued R-transforms, $R_a(z)=zR'_a(z)$, $R_b(z)=zR'_b(z)$ in $z \mathbb{K}\langle\langle z \rangle\rangle$. Recall that $zS(R'_a)(z)=R_a(z)^{\circ-1}$. First, recalling that in the commutative case we have that $\gg:=\rhd_l=\rhd_r$, we observe that
%
%
$$
	H_1(R'_a,R'_b)(z)=S(R'_a \gg S(R'_b))(z) 
	\qquad 
	\text{and}
	\qquad
	H_2(R'_a,R'_b)(z)=H_1(R'_b,R'_a) .
$$  
The first two identities of \eqref{eq:H1H2} then say for $R'_a(z),R'_b(z) \in 1+z\bb K\langle\langle z \rangle\rangle$ that
\begin{align}
\label{eq:new1}
\begin{aligned}
	(R'_a \boxcon  R'_b)(z) 
		&= S(R'_b \gg S(R'_a))(z) \, R'_b \circ (zS(R'_b \gg S(R'_a))(z))\\
		&= S(R'_a \gg S(R'_b))(z) \, R'_a \circ (zS(R'_a \gg S(R'_b))(z)).
\end{aligned}
\end{align}
The third identity of \eqref{eq:H1H2} implies that 
\begin{align}
\label{eq:new2}
	(R'_a \boxcon  R'_b)(z) = S(R'_a \gg S(R'_b))(z)\, S(R'_b \gg S(R'_a))(z), 
\end{align}
which corresponds to \cite[Identity 18.30, p.298]{nica_speicher_book}, with the incomplete boxed convolution defined at the level of $1+z\bb K\langle\langle z \rangle\rangle$ by $H_1(R'_a,R'_b) = R'_a \check{\boxcon} R'_b$ and $H_2(R'_a,R'_b) = R'_b \check{\boxcon} R'_a$.
\end{enumerate}

\end{rmk}


\subsection{Operator-valued free moment-cumulant relations}
\label{ssec:appear_naturally}

In this section, we start by considering the moment-cumulant relations in operator-valued free probability. We discover that the new group laws on $G_B^\inv$ presented earlier arise naturally. Recall (see for instance Speicher \cite{speicher_short_proof}) that moment-cumulant relations can be written
\begin{align*}
	M_a I 
	&= (K_a I) \circ (I + I M_a I) \\
	I M_a 
	&= (I K_a) \circ (I + I M_a I),
\end{align*}
where $M_a,K_a\in G_B^\inv$ are respectively the moment and cumulant series of a random variable $a \in \mathcal{A}$, see Definition \ref{def:op_valued_free_proba}. Using the left/right $I$-multiplication maps, 
$\lambda(F)=IF \in G_B^\dif$ respectively $\rho(F)=FI \in  G_B^\dif$, for $F \in G_B^\inv$,  allows us to write the moment-cumulant relations more compactly
\begin{align*}
	\rho(M_a) &= \rho(K_a)\circ \rho(1 + I M_a) \\
	\lambda(M_a) &= \lambda(K_a)\circ \lambda(1 + M_a I).
\end{align*}

We have seen that the composition in $G_B^\dif$ can be transported back to $G_B^\inv$ to define the two group laws $\star_l$ and $\star_r$ on $G_B^\inv$. In this way, $\lambda : (G_B^\inv, \star_l) \to (G_B^\dif, \circ)$ and $\rho : (G_B^\inv, \star_r) \to (G_B^\dif, \circ)$ are group morphisms, i.e., for $F,G\in G_B^\inv$,
\begin{align*}
	(IF)\circ(IG) &= I(F \star_l G) \\
	(FI)\circ(GI) &= (F \star_r G)I.
\end{align*}
Note that this intriguing observation was made in Frabetti's work \cite{frabetti} (see also \cite[Proposition 3.9]{T-transf}).
We can again rewrite the moment-cumulant relations:
\begin{align}
	M_a &= K_a \star_r (1 + I M_a) \label{eq:mom_cum_str}\\
	M_a &= K_a \star_l (1 + M_a I). \label{eq:mom_cum_stl}
\end{align}
We can also write $1+IM_a = M_a \rhd_l (1+I)$ and $1+M_aI = M_a \rhd_r (1+I)$, leading to 
\begin{align*}
	M_a &= K_a \star_r (M_a \rhd_l (1+I)) \\ 
	M_a &= K_a \star_l (M_a \rhd_r (1+I)).
\end{align*}
Note that these expressions resemble the ones of Proposition \ref{prop:rhd_from_rhd_l_and_rhd_r}. Indeed, isolating the $K_a$ by $\star_r$-/$\star_l$-multiplying from the left and right by $S^{r/l}(K_a)$ respectively $S^{r/l}(M_a)$ gives
\begin{align*}
	S^r(K_a) &= (M_a \rhd_l (1+I)) \star_r S^r(M_a) \\
	S^l(K_a) &= (M_a \rhd_r (1+I)) \star_l S^l(M_a),
\end{align*}
and then using Proposition \ref{prop:rhd_from_rhd_l_and_rhd_r} allows us to write
\begin{align}
	S^r(K_a) &= S^r(M_a) \sqdot' (1+I) \label{eq:Sr_K_Sr_M_sqdot}\\
	S^l(K_a) &= (1+I) \sqdot S^l(M_a).	\label{eq:Sl_K_Sl_M_sqdot}
\end{align}

\begin{rmk}
	We see the products $\sqdot$ and $\sqdot'$ appear. It is interesting because these group laws were introduced in light of the twisted multiplication formula \eqref{eq:twisted_factorisation}. Here, however, they are derived from trying to simplify the operator-valued free moment-cumulant relation. 

	Although, they are not really essential in these equations, as implied by the following lemma, which states that multiplication by $1+I$ with products $\sqdot$ and $\sqdot'$ is rather simple.
\end{rmk}

\begin{lem}
\label{lem:1pI_sqdot_commute}
For $F\in G_B^\mathrm{inv}$,
\begin{align*}
	F \sqdot (1+I) = (1+I) \sqdot F = (1+I) \cdot F \\ 
	F \sqdot' (1+I) = (1+I) \sqdot' F = F\cdot (1+I).
\end{align*}
\end{lem}
\begin{proof}
We have
\[
	F \sqdot (1+I) = (1+I) \cdot F \circ ((1+I)^{-1}I(1+I)) = (1+I) \cdot F
\]
because $1+I$ and $I$ commute, and
\[
	(1+I)\sqdot F = F \cdot (1+I)\circ (F^{-1}IF) = F \cdot (1+ F^{-1}IF) = F + IF.
\]
A similar derivation can be done for the other pair of equalities.
\end{proof}

\begin{rmk}
The same holds for any series $A = \sum_{n\ge 0} a_n I^n$, $a_n\in \bb K$, $a_0\neq 0$. For any $F\in G_B^\inv$, we have that $A\sqdot F = F\sqdot A = A \cdot F$ and $A\sqdot' F = F\sqdot' A = F \cdot A$.
\end{rmk}

From Lemma \ref{lem:1pI_sqdot_commute}, we can again transform equations \eqref{eq:Sr_K_Sr_M_sqdot} and \eqref{eq:Sl_K_Sl_M_sqdot} to obtain
\begin{align}
\label{eq:inversecumulantsmoments}
\begin{aligned}
	S^r(K_a) &= S^r(M_a) \cdot (1+I) \\
	S^l(K_a) &= (1+I) \cdot S^l(M_a).
\end{aligned}
\end{align}
Hence, we see that the S-transforms of $K_a$ and $M_a$, i.e., cumulants versus moments, are very similar -- as one would expect from the scalar-valued analog \eqref{scalarmoments} \cite[see (16.31), p.270]{nica_speicher_book}.
Also, note that the first equality in \eqref{eq:inversecumulantsmoments} is used in \cite{speicher_short_proof} to define the S-transform $S^l(K_a)$.


\subsection{More relations between the post-groups}
\label{ssec:more_relations}

Proposition \ref{prop:rhd_from_rhd_l_and_rhd_r} points at the fact that the three post-groups $(G_B^\inv, \cdot, \rhd_l)$, $(G_B^\inv, \overline\cdot, \rhd_r)$ and $(G_B^\inv, \cdot, \rhd)$ are somehow related. In this section, we explain three other ways in which these post-groups are related. We note that these relations are not fully understood, and we believe that further studies may lead to a depper understanding of these groups. We start by noting that the previous post-groups can be seen as sub-post-groups of a larger post-group.

\begin{prop}
Consider the group $(\mathcal G, \times) := (G_B^\mathrm{inv}, \overline \cdot)\times (G_B^\mathrm{inv}, \cdot)$, and define on $\mathcal G$ the post-group action
\[{}
	(F,G) \rhd\!\!\!\rhd (H, K) := (H \circ (FIG), K\circ (FIG)).
\]
Then $(\mathcal G, \times, \rhd\!\!\rhd)$ is a post-group.

Moreover, the following maps are injective post-group morphisms (meaning that their images are sub-post-groups of $\mathcal G$, isomorphic to the first post-group):
\begin{itemize}
	\item $(G_B^\mathrm{inv}, \cdot, \rhd_l) \to (\mathcal G, \times, \rhd\!\!\rhd), \qquad F\mapsto (1,F)$,
	\item $(G_B^\mathrm{inv}, \overline\cdot, \rhd_r) \to (\mathcal G, \times, \rhd\!\!\rhd), \qquad F\mapsto (F,1)$,
	\item $(G_B^\mathrm{inv}, \cdot, \rhd) \to (\mathcal G, \times, \rhd\!\!\rhd), \qquad F\mapsto (F^{-1},F)$.
\end{itemize}
\end{prop}

\begin{proof}
All statements follow from simple computations.
\end{proof}

\begin{rmk}
\begin{enumerate}
	\item Note that it is not surprising that the two post-groups over $G_B^\inv$ can be seen as sub-post-groups of a post-group over $(G_B^\inv)^2$, since the direct product of two post-groups is still a post-group. However, it is surprising that three of them can fit inside a post-group over the product $(G_B^\inv)^2$.
	\item The map $(F,G) \in (\mathcal G, \times) \mapsto FIG \in G_B^\dif$ is a group morphism. Its kernel is 
$$
	\{a_0(1+IH, a_0^{-1}(1+HI)^{-1}) \mid H\in \Mult[[B]], a_0\in \bb K^\times\}.
$$ 
\end{enumerate}
\end{rmk}

Let us consider, for $n,m\in \bb Z$, the following maps :
\[
\begin{array}{ccccc}
	\Psi_{n,m} 	&: 	& G_B^\inv 	& \to 			& G_B^\inv \\ 
				&	& F 			& \mapsto 	& F\circ (F^nIF^m)^{\circ-1} \\ 
	\Psi'_{n,m} 	&: 	& G_B^\inv 	& \to 			& G_B^\inv \\ 
				&	& F 			& \mapsto 	& F^{-1}\circ (F^nIF^m)^{\circ-1}
\end{array}
\]
Then the following equalities can be verified, for $m, n, p, q \in \bb Z$ :
\begin{align*}
	\Psi_{m,n} \circ \Psi_{p,q} 
	&= \Psi_{m+p,n+q}   &   \Psi_{m,n}\circ \Psi'_{p,q} = \Psi'_{-m+p, -n+q} \\ 
	\Psi'_{m,n} \circ \Psi_{p,q} 
	&= \Psi'_{m+p,n+q}   &   \Psi'_{m,n}\circ \Psi'_{p,q} = \Psi_{-m+p, -n+q}. \\ 
\end{align*}

We get the following proposition as a consequence:

\begin{prop}
The set $\Gamma_\Psi := \{\Psi_{m,n}, m,n\in \bb Z\} \cup \{\Psi'_{m,n}, m,n\in \bb Z\}$ is a subgroup of the group of bijections of $G_B^\mathrm{inv}$. It is isomorphic to a semi-direct product $\bb Z^2 \rtimes \bb Z/2\bb Z$. Its subset $\{\Psi_{m,n}, m,n\in \bb Z\}$ is a subgroup isomorphic to $\bb Z^2$.
\end{prop}

Moreover, the inverse maps of the group products, $\cdot$, $\star_l$, $\star_r$, $\sqdot$, $\sqdot'\, $, and $\boxcon$, are all among this group:

\begin{prop}
\label{prop:S_l_r_sqdot_related}
We have
\begin{align*}
	\sigma &= \Psi'_{0,0}   &   S^{\tiny\boxcon} &= \Psi'_{1,1} \\ 
	S^l &= \Psi'_{1,0}   &   S^{\sqdot} &= \Psi'_{-1,1} \\ 
	S^r &= \Psi'_{0,1}   &   S^{\sqdot'} &= \Psi'_{1,-1}
\end{align*}
In particular, $\Gamma_\Psi$ is generated by $\sigma, S^l, S^r$, and
\begin{align*}
	S^{\sqdot} = \sigma S^r S^l = S^l S^r \sigma \\ 
	S^{\sqdot'} = \sigma S^l S^r = S^r S^l \sigma \\ 
	S^{\tiny\boxcon} = S^r \sigma S^l = S^l \sigma S^r 
\end{align*}
\end{prop}

\begin{rmk}
\label{rmk:proof2} 
At this point, we may return to Proposition \ref{prop:zetaboxconvol} and present another proof using the last proposition. Indeed, it implies that
\[
	S^{\sqdot}(1+I) = S^lS^r((1+I)^{-1}) = S^l(\zeta).
\]
From \eqref{eq:Sl_K_Sl_M_sqdot} and Lemma \ref{lem:1pI_sqdot_commute},
\[
	S^l(K_a) = (1+I) \sqdot S^l(M_a) = S^l(M_a) \sqdot (1+I).
\]
Therefore
\[
	S^l(M_a) = S^l(K_a) \sqdot S^{\sqdot}(1+I) = S^l(K_a) \sqdot S^l(\zeta) = S^l(K_a \boxcon \zeta),
\]
and thus $M_a = K_a \boxcon \zeta$. Note that without using Lemma \ref{lem:1pI_sqdot_commute} we also get $M_a = \zeta \boxcon K_a$, and $\zeta$ is in the center of $(G_B^\inv, \boxcon)$.
\end{rmk}

It may be reasonable to assume that any equality involving $\rhd_l$ and $\rhd_r$ can be derived using only the post-group axioms (in some sense, that would mean that $\rhd_l$ and $\rhd_r$ are "free"). However, there is at least one statement that does not seem to fit in this category. It is given by the following Lemma.

\begin{lem}
\label{lem:rhd_r_l_additional_eq}
Let $F,G,H\in G_B^\mathrm{inv}$. Then we have
\begin{equation}
\label{eq:rhd_r_l_additional_eq}
	F \rhd_r \big((S^r(F)\rhd_r G) \rhd _l H\big) = G \rhd_l \big((S^l(G)\rhd_l F) \rhd_r H\big).
\end{equation}
\end{lem}

\begin{proof}
A short computation shows that both sides equal $H\circ(FIG)$.
\end{proof}

\begin{rmk}
Let $F,G\in G_B^\inv$, and define $\widetilde F := S^l(G)\rhd_l F$ and $\widetilde G := S^r(F)\rhd_r G$, so that the last  lemma can be written $F \rhd_r (\widetilde G \rhd _l H) = G \rhd_l (\widetilde F \rhd_r H)$. Then one may wonder whether $F$ and $G$ can be written as functions of $\widetilde F$ and $\widetilde G$. The answer is quite interesting, in fact, we have 
\begin{align*}
	G &= F \rhd_r \widetilde G\\ 
	 F &= G \rhd_l \widetilde F,
\end{align*}
and 
\begin{align*}
G &= S^l(\widetilde G\rhd_r S^l(\widetilde F)) = H_1(\widetilde F,\widetilde G)\\ 
F &= S^r(\widetilde F\rhd_l S^r(\widetilde G)) = H_2(\widetilde F,\widetilde G).
\end{align*}
Here, the subordination functions, $H_1$ and $H_2$, defined in \eqref{eq:def_H1H2} in Definition \ref{def:H1H2} appear again.
\end{rmk}

The following proposition points at the fact that most of the work done for the actions $\rhd_l$ and $\rhd_r$ could be done at the general post-group level, given the equality from the above lemma.

\begin{prop}
\label{prop:general_constr_3rd_post_group}
Let $(G,\cdot, \rhd_1)$ and $(G, \overline\cdot, \rhd_2)$ be two post-groups. Assume that for all $x,y,z\in G$,
\begin{equation}
\label{eq:assumption_eq_rhd_1_2}
	x \rhd_1 \big((S_1(x)\rhd_1 y) \rhd_2 z\big) = y \rhd_2 \big((S_2(y)\rhd_2 x) \rhd_1 z\big),
\end{equation}
where $S_1$ and $S_2$ are the respective inverse maps for $*_{\rhd_1}$ and $*_{\rhd_2}$. Then the action
\[
	x \rhd_3 y := x \rhd_1 (S_1(x) \rhd_2 y)
\]
defines a post-group $(G,\cdot, \rhd_3)$.
\end{prop}

\begin{proof}
Let us derive some equalities first. For $x,y\in G$, define $H_1 = H_1(x,y) = S_1(y \rhd_2 S_1(x))$ and $H_2 = H_2(x,y) = S_2(x \rhd_1 S_2(y))$. From the assumption applied to $S_1(x)$, $S_2(y)$ and any $z\in G$, we have
\begin{align*}
	S_1(x) \rhd_1 (S_2(H_2) \rhd_2 z) &= S_2(y) \rhd_2 (S_1(H_1) \rhd_1 z).
\end{align*}
Therefore, by substituing $z = H_2$, the left-hand side reads
\begin{align*}
	S_1(x) \rhd_1 (S_2(H_2) \rhd H_2) 
	&= S_1(x) \rhd_1 S_2(H_2)^{-1} \\ 
	&= S_1(x) \rhd_1 (x \rhd_1 S_2(y))^{-1} \\
	&= S_1(x) \rhd_1 (x \rhd_1 S_2(y)^{-1}) \\ 
	&= S_2(y)^{-1} \\
	&= S_2(y) \rhd_2 y.
\end{align*}
Therefore we can identify on the right-hand side $y = S_1(H_1) \rhd_1 H_2$, or equivalentely
\[
	H_2 = H_1 \rhd_1 y.
\]
One can derive, in a similar manner, $H_1 = H_2 \rhd_2 x$.

Then, it can be derived from the assumption \eqref{eq:assumption_eq_rhd_1_2} applied to $S_1(H_1)$, $x$ and $z$ that 
\begin{align}
	y \rhd_1 (S_1(x) \rhd_2 z) 
	&= S_1(H_1) \rhd_2 ((H_1\rhd_1 y) \rhd_1 z) \nonumber\\ 
	&= S_1(H_1) \rhd_2 (H_2 \rhd_1 z). \label{eq:eq_rhd_1_2_v2}
\end{align}

We also have
\begin{align}
	S_1(y) \cdot S_1(y) \rhd_3 S_1(x) 
		&= S_1(H_1 *_{\rhd_1} y) \label{eq:rhd_3_from_H_1}\\
	S_1(y) \cdot S_1(y) \rhd_3 S_1(x) 
		&= S_1(H_2 *_{\rhd_2} x) \label{eq:rhd_3_from_H_2}.
\end{align}
Indeed, the first equality comes from direct computation
\begin{align*}
	S_1(H_1 *_{\rhd_1} y) 
	&= S_1(y) *_{\rhd_1} S_1(H_1) \\ 
	&= S_1(y) \cdot S_1(y) \rhd_1 S_1(H_1) \\ 
	&= S_1(y) \cdot S_1(y) \rhd_1 (y \rhd_2 S_1(x)) \\ 
	&= S_1(y) \cdot S_1(y) \rhd_3 S_1(x).
\end{align*}
The second equality comes from the fact that $H_2 *_{\rhd_2} x = H_1 *_{\rhd_1} y$. Indeed,
\begin{align*}
	H_2 *_{\rhd_2} x 
	&= H_2 \overline \cdot H_2\rhd_2 x \\ 
	&= H_2\rhd_2 x \cdot H_2 \\ 
	&= H_1\cdot H_2 \\ 
	&= H_1 \cdot H_1 \rhd_1 y \\ 
	&= H_1 *_{\rhd_1} y.
\end{align*}

Let us now prove that $(G,\cdot, \rhd_3)$ is a post-group. It suffices to show that, for $x,y,z\in G$, we have
\begin{equation}
\label{eq:change_var_rhd_3_post-gr}
	S_1(y) \rhd_3 (S_1(x) \rhd_3 z) = (S_1(y)\cdot S_1(y)\rhd_3 S_1(x)) \rhd_3 z.
\end{equation}
It is equivalent to the usual post-group axiom by a change of variables, and will make notations lighter. We have
\begin{align*}
	S_1(y) \rhd_3 (S_1(x) \rhd_3 z) &= S_1(y) \rhd_1 (y \rhd_2 (S_1(x) \rhd_1 (x \rhd_2 z))) \\ 
	&= S_1(y) \rhd_1 (S_1(H_1) \rhd_1 (H_2 \rhd_2 (x \rhd_2 z))) \\ 
	&= (S_1(y) *_{\rhd_1} S_1(H_1)) \rhd_1 ((H_2 *_{\rhd_2} x) \rhd_2 z) \\ 
	&= S_1(H_1 *_{\rhd_1} y) \rhd_1 ((H_2 *_{\rhd_2} x) \rhd_2 z) \\ 
	&= (S_1(x)\cdot S_1(y)\rhd_3 S_1(x)) \rhd_3 z,
\end{align*}
where the second equality comes from \eqref{eq:eq_rhd_1_2_v2}, and the last equality comes from \eqref{eq:rhd_3_from_H_1} and \eqref{eq:rhd_3_from_H_2}.
\end{proof}

\begin{rmk}
\begin{enumerate}
	\item Proposition \ref{prop:general_constr_3rd_post_group} generalises the setting of the two post-groups $(G_B^\inv, \cdot, \rhd_l)$ and $(G_B^\inv, \overline\cdot, \rhd_r)$. It shows that the construction of the post-group $(G_B^\inv, \cdot, \sqdot)$ is general, given identity \eqref{eq:assumption_eq_rhd_1_2}, which is in this case given by Lemma \ref{lem:rhd_r_l_additional_eq}.
	
	\item We can reverse the roles of the post-groups $(G,\cdot, \rhd_1)$ and $(G,\overline\cdot, \rhd_2)$ in the above proposition. Therefore there is a also post-group $(G, \overline\cdot, \rhd_4)$, with $x \rhd_4 y = x \rhd_2 (S_2(x) \rhd_1 y)$. However, we can show that $x\rhd_4 y = x^{-1} \rhd_3 y$. Hence the post-groups $(G, \cdot, \rhd_3)$ and $(G, \overline\cdot, \rhd_4)$ are related through the bijection of Remark \ref{rmk:post_group_corresp_inverse_group}, and $x\mapsto x^{-1}$ is a group isomorphism between them.
	
	\item Given any post-group $(G, \cdot, \rhd_1)$, one can construct the post-group $(G, \overline\cdot, \rhd_1')$ as in Remark \ref{rmk:post_group_corresp_inverse_group}, and these two groups satisfy the assumption \eqref{eq:assumption_eq_rhd_1_2} of Proposition \ref{prop:general_constr_3rd_post_group} if and only if the group $(G, \cdot)$ is abelian, i.e., if $(G, \cdot, \rhd_1)$ is a pre-group. However, in this case, the third post-group (and thus, pre-group) given by Proposition \ref{prop:general_constr_3rd_post_group} is the trivial post-group, i.e., $x \rhd_3 y = y$.
\end{enumerate}
\end{rmk}


\section{Pre-Lie and post-Lie structures on formal multilinear function series}
\label{sec:lie_structure}


\subsection{Linearisation of composition}
\label{ssec:lin_compo_pre-Lie}

\begin{defn}
Define $\mathfrak g_B^\inv := \{f\in \Mult[[B]] \mid f_0=0\}$, as well as the bijective exponential, $\exp : \mathfrak g_B^\inv \to G_B^\inv$, with respect to multiplication $\cdot$. We can think of $\mathfrak g_B^\inv$ as the Lie algebra of $G_B^\inv$.
\end{defn}

While the product $\cdot$ on $\Mult[[B]]$ is bilinear, this is not the case for composition $\circ$. Let us define a linearisation of the latter.

\begin{defn}
For $f,g\in \mathfrak g_B^\inv$ and their composition, $f \circ g$, define the linearisation of the latter to be the series $\{f,g\}\in \mathfrak g_B^\inv$ such that for $N\ge 1$,
\begin{align*}
	\{f,g\}_N(x_1, x_2, \ldots, x_N) = \sum_{\substack{n,m\ge 1 \\ n + m = N+1}} 
	& f_n(g_m(x_1, \ldots, x_m), x_{m+1}, \ldots, x_N) + f_n(x_1, g_m(x_2, \ldots, x_{m+1}), \ldots, x_N) + \\
	& \cdots + f_n(x_1, x_2, \ldots, g_m(x_n, \ldots,x_N)).
\end{align*}
\end{defn}

Symbolically, we can write $\{f,g\}\in \mathfrak g_B^\inv$ as
\[
	\{f,g\} = \sum_{n,m\ge 1} f_n(g_m, I, \ldots, I) + \cdots + f_n(I, \ldots, I, g_m).
\]
This bracket is the linearization of composition in the following sense: For $f,g\in \mathfrak g_B^\inv$ and $s,t\in \bb K$, then $\{f,g\}$ is the coefficient of $st$ in $(sf)\circ(tg)$.

Given a vector space $V$ with bilinear product, $\diamond: V \times V \to V$, recall the definition of associator $\mathrm{a}_\diamond: V \times  V \times V \to V$, which measures the lack of associativity
\begin{equation}
\label{def:asso}
	\mathrm{a}_\diamond(x,y,z):= x \diamond (y \diamond z) -  (x \diamond y) \diamond z.
\end{equation}

\begin{prop}
The map $(f,g) \mapsto \{f,g\}$ defines a right pre-Lie product over $\mathfrak g_B^\mathrm{inv}$, i.e. for $f,g,h\in \mathfrak g_B^\mathrm{inv}$, it satisfies the right pre-Lie identity $\mathrm{a}_{\scriptscriptstyle{\{\cdot,\cdot\}}}(f,g,h) = \mathrm{a}_{\scriptscriptstyle{\{\cdot,\cdot\}}}(f,h,g)$, or explicitly
\[
	\{f, \{g, h\}\} - \{\{f, g\}, h\} = \{f, \{h, g\}\} - \{\{f, h\}, g\}.
\]
\end{prop}
 
 We refer the reader to \cite{burde2006,manchon2011} for reviews on pre-Lie algebras, including examples and applications.

\begin{proof}
Let $f,g,h \in \mathfrak g_B^\inv$. On the one hand
\begin{align*}
	\{f, \{g, h\}\} 
	&= \sum_{n,n'\ge 1} f_n(\{g,h\}_{n'}, I, \ldots, I) + \cdots + f_n(I, \ldots, I, \{g,h\}_{n'}) \\ 
	&= \sum_{n,m,p\ge 1} \Big(f_n(g_m(h_p, I, \ldots, I), I, \ldots, I) + \cdots + f_n(g_m(I, \ldots, I, h_p), I, \ldots, I)\Big) \\ 
	& \hspace{40pt}  + \cdots + \Big(f_n(I, \ldots, I, g_m(h_p, I, \ldots, I)) + \cdots + f_n(I, \ldots, I, g_m(I, \ldots, I, h_p))\Big),
\end{align*}
where the inside sum is over all possible position of $h_p$ inside $g_m$ inside $f_n$; there are $nm$ terms. On the other hand,
\begin{align*}
	\{\{f, g\}, h\} 
	&= \sum_{p',p\ge 1} \{f,g\}_{p'}(h_p, I, \ldots, I) + \cdots + \{f,g\}_{p'}(I, \ldots, I, h_p) \\ 
	&= \sum_{n,m,p\ge 1} \Big( f_n(g_m(h_p, I, \ldots, I), I, \ldots, I) + f_n(g_m(I, \ldots, I), I, \ldots, I, h_p)\Big) \\
	& \hspace{40pt} + \Big(f_n(h_p, g_m(I, \ldots, I), I, \ldots, I) + f_n(I, g_m(I, \ldots, I), I, \ldots, I, h_p)\Big)\Big) \\ 
	& \hspace{40pt} + \cdots + \Big(f_n(h_p, I, \ldots, I, g_m(I, \ldots, I)) + \cdots + f_n(I, \ldots, I, g_m(I, \ldots, I, h_p))\Big).
\end{align*}
Note that this time, in the inside sum, the $h_p$ can be anywhere inside the $f_n$, including outside $g_m$. There are $n(n+m-1)$ terms; $n$ choices for $g_m$, and $n+m-1$ choices for $h_p$. Therefore, in $\{\{f, g\}, h\} - \{f, \{g, h\}\}$, all terms from $\{f, \{g, h\}\}$ cancel out, leaving only the $n(n-1)$ terms where $h_p$ is not inside of $g_m$. That is,
\begin{align*}
	\{\{f, g\}, h\} - \{f, \{g, h\}\} 
	&= \sum_{n,m,p\ge 1} f_n(g_m, h_p, I, \ldots, I) + \cdots + f_n(g_m, I, \ldots, I, h_p) \\
	& \hspace{40pt} + f_n(h_p, g_m, I, \ldots, I) + f_n(I, g_m, h_p, I, \ldots, I) + \cdots + f_n(I, g_m, I, \ldots, I, h_p)  \\ 
	& \hspace{40pt} + \cdots + f_n(h_p, I, \ldots, g_m) + \cdots + f_n(I, \ldots, I, h_p, g_m).
\end{align*}
This expression is now symmetric in $g$ and $h$, which concludes the proof.
\end{proof}

The following lemma can be understood as the Lie algebra analog of the right distributivity of composition over multiplication \eqref{rightaction} stated in Lemma \ref{lem:rightaction}. It has multiple consequences.

\begin{lem}
\label{lem:lin_compo_derivation_left}
For $f,g,h\in \mathfrak g_B^\mathrm{inv}$, we have the equality
\begin{equation}
\label{eq:lin_compo_derivation_left}
	\{f \cdot g,h\} = \{f,h\} \cdot g + f \cdot  \{g,h\}.
\end{equation}
\end{lem}

\begin{proof}
Let $f,g,h\in \mathfrak g_B^\inv$. Then
\begin{align*}
	\{f \cdot g,h\} 
	&= \sum_{p',p\ge 1} (f \cdot g)_{p'}(h_p, I, \ldots, I) + \cdots + (f \cdot g)_{p'}(I, \ldots, I, h_p) \\ 
	&= \sum_{n,m,p\ge 1} f_n(h_p, I, \ldots, I) \cdot g_m(I, \ldots, I) + \cdots + f_n(I, \ldots, I, h_p) \cdot g_m(I, \ldots, I) \\ 
	& \hspace{40pt} + f_n(I, \ldots, I) \cdot g_m(h_p, I, \ldots, I) + \cdots + f_n(I, \ldots, I) \cdot g_m(I, \ldots, I, h_p) \\
	&= \sum_{n,m,p\ge 1} \Big(f_n(h_p, I, \ldots, I) + \cdots + f_n(I, \ldots, I, h_p)\Big)\cdot g_m \\ 
	& \hspace{40pt} + f_n\cdot \Big(g_m(h_p, I, \ldots, I) + \cdots + g_m(I, \ldots, I, h_p)\Big) \\ 
	&= \{f,h\} \cdot g + f \cdot \{g,h\}.
\end{align*}
\end{proof}

\begin{coroll}
\label{coroll:derivation_lin_compo_I}
For $f,g\in \mathfrak g_B^\mathrm{inv}$, we have the equalities
\begin{align*}
	\{I \cdot f,g\} &= I \cdot \{f,g\} + g \cdot f \\ 
	\{f \cdot I,g\} &= \{f,g\} \cdot I + f \cdot g.
\end{align*}
\end{coroll}
\begin{proof}
It follows immediately from Lemma \ref{lem:lin_compo_derivation_left} and the fact that $\{I,g\} = g$.
\end{proof}

\begin{rmk}
Note that, however, $\{f,I\}\neq f$ in general. In fact, $\{-,I\}$ is the grading operator: If $f$ is homogenous of degree $n$, i.e., if $f$ identifies with a $n$-linear map, then $\{f,I\} = nf$.
\end{rmk}

We introduce now the notion of Nijenhuis operator on a vector space with a bilinear product (later specified to be a pre-Lie algebra). Nijenhuis operators on associative algebras appeared in \cite{CGM2000}.

\begin{defn}
\label{def:Nijenhuismap}
Let $(A,\diamond)$ be a $\mathbb{K}$-vector space equipped with a bilinear product. A $\mathbb{K}$-linear map $N: A \to A$ is called Nijenhuis operator if it satisfies the Nijenhuis identity for all $x,y \in A$
\begin{equation}
\label{eq:Nijenhuismap}
	N(x) \diamond N(y) = N\big(N(x) \diamond y + x \diamond N(y) - N(x \diamond y) \big).
\end{equation}
\end{defn}

\begin{prop}
\label{prop:mult_by_I_Nijenhuis}
The maps $f \mapsto \lambda(f)=I \cdot f$ and $f \mapsto \rho(f)=f \cdot I$ are Nijenhuis maps for the pre-Lie algebra $(\mathfrak g_B^\mathrm{inv}, \{\cdot,\cdot\})$, i.e., they satisfy the relations
\[
	\{\lambda(f),\lambda(g)\} = \lambda\big( \{\lambda(f),g\} + \{f,\lambda(g)\} - \lambda(\{f,g\}) \big),
\]
and similarly for $\rho$.
\end{prop}

\begin{proof}
Let us derive the equality for $\lambda(f)=If$, the other one being similar. Corollary \ref{coroll:derivation_lin_compo_I} applied to $(f, g)$ and $(f, Ig)$ gives
\begin{align}
\label{twoidentities}
\begin{aligned}
	\{I \cdot f,g\} &= I\{f,g\} + g \cdot f \\ 
	\{I \cdot f,I \cdot g\} &= I\{f,I \cdot g\} + (I \cdot g) \cdot f.
\end{aligned}
\end{align}
Multiplying the first equality by $I$, i.e., applying $\lambda$ leads to 
\[
	I\{I \cdot f,g\} = I^2\{f,g\} + I \cdot g \cdot f.
\]
We get the result by subtracting this equality from the second equality in \eqref{twoidentities}:
\[
	\{If,Ig\} = I\{f,Ig\} + I\{If,g\} - I^2\{f,g\}.
\]
\end{proof}

Let us also note the following corollary from Lemma \ref{lem:lin_compo_derivation_left}, where the Lie bracket for $f,g \in  \mathfrak g_B^\mathrm{inv}$ is defined by $[f,g]=f \cdot g - g \cdot f$.

\begin{coroll}
\label{coroll:first_eq_post_Lie_lin_compo}
For $f,g,h\in \mathfrak g_B^\mathrm{inv}$, we have
\[
	\{[f,g],h\} = [f, \{g,h\}] + [\{f,h\},g].
\]
\end{coroll}

Returning to Proposition \ref{prop:mult_by_I_Nijenhuis}, we consider both Nijenhuis maps simultaneously which leads to another equality, that will be useful later.

\begin{lem}
\label{lem:mix_both_Nijenhuis}
Let $f,g\in \mathfrak g_B^\inv$. Then
\begin{equation}
\label{eq:mixedid}
	\{g \cdot I, I \cdot f\} - \{I \cdot f, g \cdot I\} = \{g, I \cdot f\}I - I\{f, g \cdot I\}.
\end{equation}
\end{lem}

\begin{proof}
From Corollary \ref{coroll:derivation_lin_compo_I}, we obtain
\begin{align*}
	\{g\cdot I, I\cdot f\} &= \{g, I\cdot f\}I + g\cdot I\cdot f \\
	\{I\cdot f, g\cdot I\} &= I\cdot \{f, g\cdot I\} + g\cdot I\cdot f.
\end{align*}
Subtracting both equalities gives the desired identity \eqref{eq:mixedid}.
\end{proof}


\subsection{Post-Lie algebra structures}
\label{ssec:post-Lie}

We will see that three post-Lie algebras can be constructed from the pre-Lie product $\{f,g\}$ defined over $\mathfrak g_B^\mathrm{inv}$. Vallette introduced the notion of post-Lie algebra in \cite{vallette2007}. The closely related notion of $D$-algebra appeared independently in \cite{MKW2008}. Both pre- and post-Lie algebras are naturally defined in terms of invariant connections defined on manifolds. See \cite{curry2019} for a concise summary, including examples and application. 

\begin{defn}
\label{def:postLie}
A \emph{post-Lie algebra} $(\mathfrak{g}, [\cdot,\cdot], \blacktriangleright)$ consists of a Lie algebra $(\mathfrak g, [\cdot,\cdot])$ together with a linear product $\blacktriangleright: {\mathfrak g} \otimes {\mathfrak g} \rightarrow \mathfrak g$ such that the following relations hold for all elements $x,y,z \in \mathfrak g$
\begin{align}
\label{postLie1}
	z \blacktriangleright[x,y] 	
	&= [z \blacktriangleright x , y] + [x , z \blacktriangleright y]\\
\label{postLie2}
	[x,y] \blacktriangleright z 	
	&= {\mathrm{a}}_{\blacktriangleright}(x,y,z) - {\mathrm{a}}_{\blacktriangleright}(y,x,z).
\end{align}
where the associator ${\mathrm{a}}_{\blacktriangleright}(x,y,z)$ of $x,y,z$ is defined in \eqref{def:asso}.

An abelian post-Lie algebra, i.e., when $\mathfrak{g}$ is a trivial Lie algebra, reduces to a (left) pre-Lie algebra $(\mathfrak{g}, \blacktriangleright )$ with left pre-Lie identity ${\mathrm{a}}_{\blacktriangleright}(x,y,z) = {\mathrm{a}}_{\blacktriangleright}(y,x,z).$
\end{defn}

The following result is standard in the theory of post-Lie algebras \cite{curry2019}.

\begin{thm}
In any post-Lie algebra $(\mathfrak{g}, [\cdot,\cdot], \blacktriangleright)$, there is a second Lie algebra structure available, given by the bracket
\[
	[x,y]_\blacktriangleright := [x,y] + x \blacktriangleright y - y \blacktriangleright x.
\]
In particular, a pre-Lie algebra is Lie admissible, i.e., $[x,y]_\blacktriangleright = x \blacktriangleright y - y \blacktriangleright x$ is a Lie bracket.
\end{thm}

A correspondence between post-Lie groups (that is, post-groups that are also Lie groups) and post-Lie algebras is given in \cite{Bai_Guo_Sheng_Tang_post_groups} (see also \cite{al-kaabi_ebrahimi-fard_manchon}). For a post-Lie group $(G, \cdot, \rhd)$, a post-Lie action is defined on $(\mathfrak g, [\cdot,\cdot])$ by
\begin{equation}
\label{eq:post-Lie_exp}
	x \blacktriangleright y = \frac{d^2}{dsdt} \Big|_{s,t=0} (\exp(tx) \rhd \exp(sy)). 
\end{equation}

Regarding $(G_B^\inv, \cdot)$, we can find several post-Lie algebra structures over its Lie algebra $(\mathfrak g_B^\inv, [\cdot,\cdot])$. Indeed, using formula \eqref{eq:post-Lie_exp}, from the post-group action $\rhd_l$ we get for $f,g\in \mathfrak g_B^\inv$
\begin{align*}
	\exp(tf) \rhd_l \exp(sg) &= \exp(sg) \circ (I \exp(tf)) \\ 
	&= (1 + sg + o(s)) \circ (I(1+tf+o(t))) \\ 
	&= (1 + sg + o(s)) \circ (I + t If + o(t)),
\end{align*}
and the coefficient of $st$ in this expression will be the multilinear series $\{g, If\}$. Similarly, from $\rhd_r$ we deduce $\{g,fI\}$. These define indeed post-Lie products.

\begin{prop}\footnote{This observation was first made by N.~Gilliers in the context of the work \cite{T-transf}.}
\label{prop:black_triangles_are_post_Lie}
The following binary operations are post-Lie products over $(\mathfrak g_B^\mathrm{inv}, [-, -])$:
\begin{align*}
	f \blacktriangleright_l g 	&:= \{g, If\} \\ 
	f \blacktriangleright_r g 	&:= \{g, -fI\} = -\{g,fI\}. 
\end{align*}
Combining the above post-Lie products, we find another post-Lie product
\begin{equation}
\label{eq:3rdpost-Lie_prod}
	f \blacktriangleright g := \{g, [I,f]\}.
\end{equation}
\end{prop}

\begin{proof}
From Corollary \ref{coroll:first_eq_post_Lie_lin_compo}, we deduce for $f,g,h\in \mathfrak g_B^\mathrm{inv}$
\[
	h \blacktriangleright_l[f,g] 
	= \{[f,g],Ih\} 
	= [f, \{g,Ih\}] + [\{f,Ih\},g] 
	= [f , h \blacktriangleright_l g] + [h \blacktriangleright_l f, g].
\]
Moreover,
\begin{align*}
	\mathrm{a}_{\blacktriangleright_l}(f,g,h) - \mathrm{a}_{\blacktriangleright_l}(g,f,h) 
	&= f \blacktriangleright_l (g \blacktriangleright_l h) - (f \blacktriangleright_l g) \blacktriangleright_l h 
		- g \blacktriangleright_l (f \blacktriangleright_l h) + (g \blacktriangleright_l f) \blacktriangleright_l h \\ 
	&= \{\{h, Ig\}, If\} - \{h, I\{g, If\}\} - \{\{h, If\}, Ig\} + \{h, I\{f, Ig\}\} \\ 
	&= \{h, \{Ig, If\}\} - \{h, I\{g, If\}\} - \{h, \{If, Ig\} + \{h, I\{f, Ig\}\} \\ 
	&= \{h, Ifg\} - \{h, Igf\} \\ 
	&= \{h, I[f,g]\} \\ 
	&= [f,g] \blacktriangleright_l h.
\end{align*}
The third equality is a consequence of $\{\cdot,\cdot\}$ being pre-Lie, and the fourth one comes from Corollary \ref{coroll:derivation_lin_compo_I}. Therefore $(\mathfrak g_B^\inv, [-, -], \blacktriangleright_l)$ is a post-Lie algebra. The verification is similar for the product $\blacktriangleright_r$. 

For the product \eqref{eq:3rdpost-Lie_prod}, equation \eqref{postLie1} is easily checked. Then, we have, for $x,y,z\in \mathfrak g_B^\inv$,
\begin{align*}
	\lefteqn{\mathrm{a}_\blacktriangleright(f,g,h) - \mathrm{a}_\blacktriangleright(g,f,h) 
	= f \blacktriangleright_l (g \blacktriangleright_l h)  + f \blacktriangleright_r (g \blacktriangleright_r h) 
		+ f \blacktriangleright_l (g \blacktriangleright_r h) + f \blacktriangleright_r (g \blacktriangleright_l h)}\\
	&- (f \blacktriangleright_l g) \blacktriangleright_l h - (f \blacktriangleright_r g) \blacktriangleright_r h 
		- (f \blacktriangleright_l g) \blacktriangleright_r h - (f \blacktriangleright_r g) \blacktriangleright_l h \\ 
	&- g \blacktriangleright_l (f \blacktriangleright_l h) - g \blacktriangleright_r (f \blacktriangleright_r h) 
		- g \blacktriangleright_l (f \blacktriangleright_r h) - g \blacktriangleright_r (f \blacktriangleright_l h) \\ 
	&+ (g \blacktriangleright_l f) \blacktriangleright_l h + (g \blacktriangleright_r f) \blacktriangleright_r h 
		+ (g \blacktriangleright_l f) \blacktriangleright_r h + (g \blacktriangleright_r f) \blacktriangleright_l h \\ 
	&= \mathrm{a}_{\blacktriangleright_l}(f,g,h) - \mathrm{a}_{\blacktriangleright_l}(g,f,h) 
		+ \mathrm{a}_{\blacktriangleright_r}(f,g,h) - \mathrm{a}_{\blacktriangleright_r}(g,f,h) \\ 
	&\hspace{3cm}+\mathrm{a}_{\blacktriangleright_l, \blacktriangleright_r}(f,g,h) - \mathrm{a}_{\blacktriangleright_l, \blacktriangleright_r}(g,f,h) 
		+ \mathrm{a}_{\blacktriangleright_r, \blacktriangleright_l}(f,g,h) 
		 - \mathrm{a}_{\blacktriangleright_r, \blacktriangleright_l}(g,f,h),
\end{align*}
where for any two binary products $\bullet_1, \bullet_2$
\[
	\mathrm{a}_{\bullet_1, \bullet_2}(f,g,h) := f \bullet_1 (g \bullet_2 h) - (f \bullet_1 g) \bullet_2 h. 
\]
Note that 
\begin{align*}
	\mathrm{a}_{\blacktriangleright_l}(f,g,h) - \mathrm{a}_{\blacktriangleright_l}(g,f,h) 
	+ \mathrm{a}_{\blacktriangleright_r}(f,g,h) - \mathrm{a}_{\blacktriangleright_r}(g,f,h) 
	&= [f,g] \blacktriangleright_l h + [f,g] \blacktriangleright_r h \\ 
	&= [f,g] \blacktriangleright h.
\end{align*}
So we have to show that 
\begin{equation}
\label{eq:mixed_post_assoc_vanish}
	\mathrm{a}_{\blacktriangleright_l, \blacktriangleright_r}(f,g,h) 
	- \mathrm{a}_{\blacktriangleright_l, \blacktriangleright_r}(g,f,h) 
	+ \mathrm{a}_{\blacktriangleright_r, \blacktriangleright_l}(f,g,h) 
	- \mathrm{a}_{\blacktriangleright_r, \blacktriangleright_l}(g,f,h) = 0
\end{equation}
We claim that for any $f,g,h \in \mathfrak g_B^\inv$,
\begin{equation}
\label{eq:claim_post_Lie_proof}
	f \blacktriangleright_l (g\blacktriangleright_r h) - (f \blacktriangleright_l g)\blacktriangleright_r h 
	= g\blacktriangleright_r (f \blacktriangleright_l h)  - (g\blacktriangleright_r f) \blacktriangleright_l h.
\end{equation}
The claim is equivalent to $\mathrm{a}_{\blacktriangleright_r, \blacktriangleright_l}(f,g,h) - \mathrm{a}_{\blacktriangleright_l, \blacktriangleright_r}(g,f,h) = 0$, and therefore implies equality \eqref{eq:mixed_post_assoc_vanish}. Let us prove \eqref{eq:claim_post_Lie_proof}. For $f,g,h \in \mathfrak g_B^\inv$, we have
\begin{align*}
	f \blacktriangleright_l (g\blacktriangleright_r h) - g\blacktriangleright_r (f \blacktriangleright_l h) 
	&= \{\{h, gI\}, If\} - \{\{h, If\}, gI\} \\ 
	&= \{h, \{gI, If\}\} - \{h, \{If, gI\}\} \\ 
	&= \Big\{h, \{gI, If\} - \{If, gI\}\Big\} \\ 
	&= \Big\{h, \{g, If\}I - I\{f, gI\}\Big\} \\ 
	&= \{h, \{g, If\}I\} - \{h, I\{f, gI\}\} \\ 
	&= (f \blacktriangleright_l g) \blacktriangleright_r h - (g\blacktriangleright_r f) \blacktriangleright_l h.
\end{align*}
The first equality follows from the pre-Lie identity for $\{\cdot,\cdot\}$, applied to $(h, gI, If)$, and the fourth equality follows from Lemma \ref{lem:mix_both_Nijenhuis}.
\end{proof}

\begin{rmk}
\begin{enumerate}
	\item Post-Lie products over a Lie algebra $(\mathfrak g, [\cdot,\cdot])$ and its opposite Lie algebra $(\mathfrak g, -[\cdot,\cdot])$ are in bijective correspondence by a change of sign. Precisely, the product $\blacktriangleright_1$ is a post-Lie product over $(\mathfrak g, [\cdot,\cdot])$ if and only if $-\blacktriangleright_1$ is a post-Lie product over $(\mathfrak g, -[\cdot,\cdot])$. This remark is of course strongly linked to Remark \ref{rmk:post_group_corresp_inverse_group}.
	
	\item The first remark explains why the post-Lie product corresponding to $\rhd_r$ is $-\{g,fI\}$ and not $\{g,fI\}$; as $\rhd_r$ is a post-group action over $(G_B^\inv, \overline \cdot)$, the product $\{g,fI\}$ is a post-Lie product over $(\mathfrak g_B^\inv, -[\cdot,\cdot])$.
	
	\item There is no additional post-Lie product for the post-group action $\rhd'$. Indeed, the correspondence also gives the product $\blacktriangleright$, the same as for the action $\rhd$. More precisely, $\rhd'$ is a post-group action for the inverse product $\overline \cdot$ on $G_B^\inv$, thus its corresponding post-Lie algebra is $(\mathfrak g_B^\inv, -[\cdot,\cdot], -\blacktriangleright)$.
\end{enumerate}
\end{rmk}

It is somewhat unexpected that
$$
	f \blacktriangleright g = f \blacktriangleright_l g + f \blacktriangleright_r g.
$$
Post-Lie products are not in general stable by linear combinations, or even sums. The following proposition gives a condition on which that is true. 

\begin{prop}
Let $(\mathfrak g, [\cdot,\cdot], \blacktriangleright_1)$ and $(\mathfrak g, [\cdot,\cdot], \blacktriangleright_2)$ be two post-Lie algebras. Assume that for all $x,y,z\in \mathfrak g$, we have
\begin{equation}
\label{eq:identity_for_sum_post-Lie}
	x \blacktriangleright_1 (y\blacktriangleright_2 z) - (x \blacktriangleright_1 y)\blacktriangleright_2 z 
	= y\blacktriangleright_2 (x \blacktriangleright_1 z) - (y\blacktriangleright_2 x) \blacktriangleright_1 z.
\end{equation}
Then $(\mathfrak g, [\cdot,\cdot], \blacktriangleright_3)$ is a post-Lie algebra, with 
\[
	x \blacktriangleright_3 y = x \blacktriangleright_1 y + x \blacktriangleright_2 y.
\]
\end{prop}

\begin{proof}
The proof is a straightforward verification of the post-Lie algebra axioms. The assumption is used to verify \eqref{postLie2}, in a similar manner as in the end of the proof of Proposition \ref{prop:black_triangles_are_post_Lie}.
\end{proof}

The last proposition corresponds to Proposition \ref{prop:general_constr_3rd_post_group} at the level of Lie-algebras.

\begin{prop}
Let $(G, \cdot, \rhd_1)$ and $(G, \cdot, \rhd_2)$ be post-groups satisfying the condition of Proposition \ref{prop:general_constr_3rd_post_group}. Assume that both post-groups are post-Lie groups, i.e., the products $\cdot$, $\rhd_1$, $\rhd_2$ are smooth. Let $(\mathfrak g, [\cdot,\cdot], \blacktriangleright_1)$ and $(\mathfrak g, -[\cdot,\cdot], -\blacktriangleright_2)$ be their associated post-Lie algebras \cite{Bai_Guo_Sheng_Tang_post_groups}. Then identity \eqref{eq:identity_for_sum_post-Lie} holds for products $\blacktriangleright_1$ and $\blacktriangleright_2$, giving rise to the post-Lie algebra $(\mathfrak g, \cdot, \blacktriangleright_3)$. Moreover, the post-group $(G, \cdot, \rhd_3)$ given by Proposition \ref{prop:general_constr_3rd_post_group} is a post-Lie group, and its post-Lie algebra is $(\mathfrak g, \cdot, \blacktriangleright_3)$.
\end{prop}

\begin{proof}
The assumption is that for all $x,y,z\in G$,
\[
	x \rhd_1 ((S_1(x)\rhd_1 y) \rhd_2 z) = y \rhd_2 ((S_2(y)\rhd_2 x) \rhd_1 z).
\]

First, for $X, Y\in \mathfrak g$ and $i\in \{1,2\}$, we can show that
\begin{align*}
	\frac{d^2}{dsdt} \Big|_{s,t=0} S_i(e^{tX}) \rhd_i e^{sY} &= -X \blacktriangleright_i Y.
\end{align*}

Let $X,Y,Z\in \mathfrak g$. We have
\begin{align*}
	&\frac{d^3}{dsdtdr} \Big|_{s,t,r=0} e^{tX} \rhd_1 ((S_1(e^{tX})\rhd_1 e^{sY}) \rhd_2 e^{rZ}) \\
	&\hspace{30pt} = \frac{d^3}{dsdtdr} \Big|_{s,t,r=0} e^{tX} \rhd_1 (e^{sY} \rhd_2 e^{rZ}) 
		+ \frac{d^3}{dsdtdr} \Big|_{s,t,r=0} (S_1(e^{tX})\rhd_1 e^{sY}) \rhd_2 e^{rZ}	\\ 
	&\hspace{30pt} = X \blacktriangleright_1 (Y \blacktriangleright_2 Z) 
			- (X \blacktriangleright_1 Y) \blacktriangleright_2 Z.
\end{align*}
Therefore given the assumption, this expression equals, from a similar computation,
\[
	Y \blacktriangleright_2 (X \blacktriangleright_1 Z) 
						- (Y \blacktriangleright_2 X) \blacktriangleright_1 Z,
\]
and we get equality \eqref{eq:identity_for_sum_post-Lie}.

Then, recall that $x \rhd_3 y = x \rhd_1 (S_1(x) \rhd_2 y)$, and that $\rhd_3$ is automatically smooth if both $\rhd_1$ and $\rhd_2$ are. Moreover, for $X,Y\in \mathfrak g$,
\begin{align*}
	\frac{d^2}{dsdt}\Big|_{s,t=0} e^{tX} \rhd_3 e^{sY} 
	&= \frac{d^2}{dsdt}\Big|_{s,t=0} e^{tX} \rhd_1 (S_1(e^{tX}) \rhd_2 e^{sY}) \\ 
	&= \frac{d^2}{dsdt}\Big|_{s,t=0} e^{tX} \rhd_1 e^{sY} + \frac{d^2}{dsdt}\Big|_{s,t=0} S_1(e^{tX}) \rhd_2 e^{sY} \\ 
	&= X \blacktriangleright_1 Y - (-X \blacktriangleright_2 Y) \\
	&= X \blacktriangleright_3 Y.
\end{align*}
The double minus sign comes from the fact that the post-Lie algebra of $(G,\overline\cdot, \rhd_2)$ is $(\mathfrak g, -[\cdot,\cdot], -\blacktriangleright_2)$. Therefore $(\mathfrak g, [\cdot,\cdot], \blacktriangleright_3)$ is the post-Lie algebra of $(G,\cdot, \rhd_3)$.

\end{proof}


\subsection{Nijenhuis operator over a pre-Lie algebra}
\label{ssec:Nijenhuis_pre_Lie}

In this section, we explore the general setting of a pre-Lie algebra equipped with a Nijenhuis operator.

\begin{defn}
\label{def:derivator}
Let $(A,+)$ be a linear vector space, and $\bullet, \cdot$ be two bilinear products over $A$. Define the {\emph{derivator}} of three elements $x,y,z\in A$ to be 
\begin{equation}
\label{eq:derivator}
	\mathrm{Der}_{\bullet, \cdot}(x,y,z) = x\bullet(y\cdot z) - (x\bullet y)\cdot z - y \cdot (x\bullet z).
\end{equation}
The {\it{derivator}} is zero if and only if the product $\bullet$ is a deriviation over the product $\cdot$.
\end{defn}

\begin{prop}
Let $(A, \rhd)$ be a left pre-Lie algebra. Consider a Nijenhuis map $N:A\to A$ over $A$. Define for all $x,y \in A$ the new product $\cdot$ over $A$
\[
	x \cdot y := x \rhd N(y) - N(x\rhd y).
\]
Define also, for $x,y\in A$,
\[
	x \blacktriangleright y := N(x) \rhd y.
\]
Then the following statements are true:
\begin{enumerate}
	\item \label{item1} For all $x,y,z\in A$,
	\[
		\mathrm{Der}_{\rhd, \cdot}(x,y,z) = \mathrm{Der}_{\rhd, \cdot}(y,x,z).
	\]
	\item \label{item2} For all $x,y,z\in A$,
	\[
		\mathrm{a}_\cdot(x,y,z) - \mathrm{a}_\cdot(y,x,z) = \mathrm{Der}_{\rhd, \cdot}(N(y),x,z) - \mathrm{Der}_{\rhd, \cdot}(N(x),y,z).
	\]
	\item \label{item3} For all $x,y,z\in A$,
	\[
		\mathrm{Der}_{\rhd, \cdot}(N(x),y,z) = \mathrm{Der}_{\blacktriangleright, \cdot}(x,y,z).
	\]
	\item \label{item4} For all $x,y,z\in A$,
	\[
		\mathrm{a}_\blacktriangleright(x,y,z) - \mathrm{a}_\blacktriangleright(y,x,z) = [x,y]\blacktriangleright z,
	\]
	where $[x,y] := x\cdot y - y \cdot x$. 
	\item \label{item5} If for all $x,y,z\in A$, $\mathrm{Der}_{\rhd, \cdot}(x,y,z) = 0$, then $\cdot$ is a left pre-Lie product and $[\cdot,\cdot]$ is a Lie bracket. Consequently, $(A, [\cdot,\cdot], \blacktriangleright)$ is a post-Lie algebra.
\end{enumerate}
\end{prop}

\begin{proof}
The proof follows by extended computations based on the two equalities, for $x,y,z\in A$:
\begin{align*}
	\mathrm{a}_\rhd(x,y,z) &= \mathrm{a}_\rhd(y,x,z) \\
	N(x) \rhd N(y) &= N(x\rhd N(y)) + N(N(x)\rhd y) - N^2(x\rhd y).
\end{align*}

Here is an example, for item \eqref{item1}. Let $x,y,z\in A$, we have
\begin{align*}
	\lefteqn{\mathrm{Der}_{\rhd, \cdot}(x,y,z) = x \rhd (y \cdot z) - (x\rhd y) \cdot z - y \cdot (x \rhd z)} \\ 
	&= \Big(x \rhd (y \rhd N(z)) - x \rhd N(y \rhd z)\Big) - \Big( (x\rhd y) \rhd N(z) - N((x\rhd y) \rhd z)\Big) - y \cdot (x\rhd z) \\ 
	&= \Big(x \rhd (y \rhd N(z)) - x  \cdot (y \rhd z) - N(x \rhd (y\rhd z))\Big) 
		- \Big( (x \rhd y) \rhd N(z) - N((x\rhd y) \rhd z)\Big) - y \cdot (x\rhd z) \\ 
	&= \mathrm{a}_\rhd(x, y, N(z)) - N(\mathrm{a}_\rhd(x,y,z)) - x \cdot  (y \rhd z) - y \cdot (x\rhd z).
\end{align*}
The final expression is symmetric in $x$ and $y$, because $\rhd$ is a left pre-Lie product.

A remark can be made for item \eqref{item5}. The condition that for all $x,y,z\in A$, $\mathrm{Der}_{\rhd, \cdot}(x,y,z) = 0$ is equivalent to the left multiplication for $\bullet$ being a derivation over the product $\cdot$. In that case, item \eqref{item2} reduces to the pre-Lie identity for the product $\cdot$, and it follows that $[\cdot,\cdot]$ is a Lie bracket. Moreover, the two post-Lie identities for $\blacktriangleright$ are item \eqref{item4} and $\mathrm{Der}_{\rhd, \cdot}(N(x),y,z) = \mathrm{Der}_{\rhd, \cdot}(N(x),z,y)$, which is true because both sides are equal to zero.
\end{proof}


\appendix


\section{An alternative formalism: crossed morphisms}
\label{ssec:crossed_morphisms}

In this section, we develop an approach based on the notion of crossed morphism, which is slightly more general than post-groups. The reader is referred to the textbook \cite{hilgert-neeb2010} for more background on crossed morphisms, as well as \cite{KIA2023} and \cite{mencattini-quesney2021} regarding links to post-Lie algebras. In reference \cite{mencattini-quesney2021} the problem of integration of post-Lie algebras in the framework of crossed morphisms was addressed.

\begin{defn}
Let $(H, \times)$ and $(G, \cdot)$ be two groups, such that $H$ acts on $G$ on the left by automorphisms. Let us denote the action by $\hookrightarrow$. A \textit{crossed morphism} is a map $\varphi : H \to G$ satisfying the so-called crossed morphism property (relative to the action by $\hookrightarrow$) for all $a,b\in H$,
\[
	 \varphi(a \times b) = \varphi (a) \cdot \big(a \hookrightarrow \varphi (b)\big).
\]
\end{defn}

\begin{prop}
\label{prop:post-group_Id_crossed_mor}
Let $(G, \cdot, \rhd)$ be a post-group, with Grossman--Larson law $*$. The group $(G, *)$ acts by automorphisms on $G$ by $g \rhd h$. For this action the identity map $\Id : (G, *) \to (G, \cdot)$ is a crossed morphism.
\end{prop}

\begin{proof}
Since $(G, \cdot) \to (G, \cdot),~ b\to a\rhd b$ is a group automorphism for $a\in G$, we know that $a\rhd 1 = 1$. Thus the post-group equation $a \rhd (b \rhd c) = (a*b) \rhd c$ exactly says that $\rhd$ is a group action of $(G, *)$ over $G$.

It remains to show that $\Id : (G, *) \to (G, \cdot)$ is a crossed morphism. In fact for $g, h\in G$, $g * h = g \cdot (g\rhd h)$ is exactly the definition of $*$ such that $\Id(g * h ) =  \Id(g) \cdot (g\rhd \Id(h))$.
\end{proof}

Note that not all crossed morphisms come from this construction, but it is almost the case for bijective crossed morphisms.

\begin{prop}
\label{prop:bij_crossed_mor_post-group}
Let $\varphi:(H,\times) \to (G, \cdot)$ be a crossed morphism relative to the action by $\hookrightarrow$. Assume furthermore that $\varphi$ is bijective. Then there is a post-group structure $(G, \cdot, \rhd)$ with Grossman--Larson law such that $\varphi : (H, \times) \to (G, *)$ is a group isomorphism. Moreover, $\rhd$ can be given be $a \rhd b := \varphi^{-1}(a) \hookrightarrow b$.
\end{prop}

\begin{proof}
Define $a \rhd b := \varphi^{-1}(a) \hookrightarrow b$. Then for any $a \in G$, $(G,\cdot) \to (G,\cdot),~ b \mapsto a\rhd b$ is a group automorphism. Moreover, for $a,b,c \in G$,
\begin{align*}
	a\rhd (b\rhd c) &= \varphi^{-1}(a) \hookrightarrow (\varphi^{-1}(b) \hookrightarrow c) \\ 
	&= (\varphi^{-1}(a) \times \varphi^{-1}(b)) \hookrightarrow c \\ 
	&= \varphi^{-1}(a \cdot (\varphi^{-1}(a) \hookrightarrow b)\big) \hookrightarrow c \\ 
	&= (a \cdot (a\rhd b)\big) \rhd c, 
\end{align*}
thus $(G, \cdot, \rhd)$ is a post-group. Finally, for $a,b\in H$, 
\begin{align}
\label{eq:crossedmorph1}
	\varphi (a \times b) = \varphi (a) \cdot (a\hookrightarrow \varphi(b)) = \varphi(a) * \varphi(b),
\end{align}
so $\varphi : (H, \times) \to (G, *)$ is a group isomorphism.
\end{proof}

\begin{rmk}
\label{rmk:relRotaBaxter}
From \eqref{eq:crossedmorph1} and the fact that $\varphi$ is invertible, we deduce for $x,y\in H$ that
\begin{align}
\label{eq:crossedmorph2}
	\varphi^{-1}(x) \times \varphi^{-1}(y) 
	= \varphi^{-1}\big(x \cdot (\varphi^{-1}(x) \hookrightarrow y)\big),   
\end{align}
which implies that $\varphi^{-1}$ is a relative Rota--Baxter operator \cite[Definition 3.1]{JiSheZhu2021}.
\end{rmk}

\begin{coroll}
Let $\varphi:(H,\times) \to (G, \cdot)$ be a bijective crossed morphism, and let $(G, \cdot, \rhd)$ be the associated post-group. Then the map $\varphi$ factorises as 
\[\begin{array}{ccccc}
	(H, \times) & \overset{\varphi}{\longrightarrow} & (G, *) & \overset{\Id}{\longrightarrow} & (G, \cdot),	
\end{array}\]
which is the composition of a group isomorphism with a crossed morphism, thanks to Proposition \ref{prop:post-group_Id_crossed_mor}. Therefore, all bijective crossed morphisms are isomorphic to a crossed morphism coming form a post-group. Moreover, the converse is also true: Any composition of an isomorphism with a crossed morphism is again a crossed morphism.
\end{coroll}

\begin{defn}
Define the product $\boxconv$ over $G_B^{I,l}$, such that $\lambda : G_B^\inv \to G_B^{I,l}$, $F \mapsto \lambda(F)=IF$ is a group isomorphism, i.e., for $F,G\in G_B^\inv$, 
\[
	(IF) \boxconv (IG) := I (F\boxcon G).
\]
\end{defn}
\begin{rmk}
This boxed convolution operation corresponds to the one from \cite{part1}. It is the natural generalisation of the scalar-valued boxed convolution from \cite{nica_speicher_book}.
\end{rmk}

Let us now explore the crossed morphisms in our setup. First, the maps 
$$
	\Id : (G_B^\inv, \star_l) \to (G_B^\inv, \cdot),
	\qquad \Id : (G_B^\inv, \star_r) \to (G_B^\inv, \overline\cdot), 
	\qquad \Id : (G_B^\inv, \sqdot) \to (G_B^\inv, \cdot)
$$ 
are all crossed morphisms thanks to Proposition \ref{prop:post-group_Id_crossed_mor}. Moreover, we have group isomorphisms
\[
\begin{array}{ccc}
	(G_B^\inv, \boxcon) & \overset{S^l}{\longrightarrow} & (G_B^\inv, \sqdot) \\ 
	(G_B^{I,l}, \boxconv) & \overset{S^l\circ\lambda^{-1}}{\longrightarrow} & (G_B^\inv, \sqdot)
\end{array}
\]
and therefore we get two crossed morphisms
\[
\begin{array}{ccc}
	(G_B^\inv, \boxcon) & \overset{S^l}{\longrightarrow} & (G_B^\inv, \cdot) \\ 
	(G_B^{I,l}, \boxconv) & \overset{S^l\circ\lambda^{-1}}{\longrightarrow} & (G_B^\inv, \cdot)
\end{array}.
\]
We also have an isomorphism
\[
\begin{array}{cccc}
	\lambda : & (G_B^\inv, \star_l) & \longrightarrow & (G_B^{I,l}, \circ) \\ 
		& F & \mapsto & IF
\end{array}
\]
and therefore a crossed morphism
\[
\begin{array}{ccc}
	(G_B^{I,l}, \circ) & \overset{\lambda^{-1}}{\longrightarrow} & (G_B^\inv, \cdot).
\end{array}
\]

Note that we have the commutative diagram
\[\begin{tikzcd}
	(G_B^{I,l}, \boxconv) \arrow[dd, "f \mapsto f^{\circ-1}", leftrightarrow] \arrow[dr, "S^l\circ\lambda^{-1}"] & \\
	& (G_B^\inv, \cdot) \\
	(G_B^{I,l}, \circ) \arrow[ru, "\lambda^{-1}" right] & 
\end{tikzcd}
\]
where the two arrows on the right are crossed morphisms, but the vertical arrow on the left lacks structure. It can be expanded as follows:

\[\begin{tikzcd}
(G_B^{I,l}, \boxconv) \arrow[dd, "f \mapsto f^{\circ-1}", leftrightarrow] \arrow[r, "\lambda^{-1}", red] & (G_B^\inv, \boxcon) \arrow[dd, "S^l", leftrightarrow] \arrow[r, "S^l", red] & (G_B^\inv, \sqdot) \arrow[dr, "\Id", blue]& \\
&&& (G_B^\inv, \cdot) \\ 
(G_B^{I,l}, \circ) \arrow[r, "\lambda^{-1}", red] & (G_B^\inv, \star_l) \arrow[uur, "\Id", leftrightarrow] \arrow[urr, "\Id", blue] & &
\end{tikzcd}
\]
where red arrows are isomorphisms, and blue arrows are crossed morphisms. Composing a blue arrow by red arrows again gives a blue arrow. Note that this diagram has a right-counterpart, with $G_B^{I,r}$, $\overline \boxcon$, a product analog to $\boxconv$, $\star_r$, $S^r$, $\rho$, $\sqdot'$ replacing respectively, $G_B^{I,l}$, $\boxcon$, $\boxconv$, $\star_l$, $S^l$, $\lambda$, $\sqdot$. Finally, only considering the group laws on $G_B^\inv$, and adding the similar maps for $\sqdot'$ and $\star_r$ leads to the commutative diagram:
\[
\begin{tikzcd}
& (G_B^\inv, \star_l) \arrow[d, "\Id"] \arrow[dr, "\Id", blue] & \\ 
(G_B^\inv, \boxcon) \arrow[d, "S^\boxcon", red] \arrow[ur, "S^l"] \arrow[r, "S^l", red] & (G_B^\inv, \sqdot) \arrow[r, "\Id", blue] \arrow[d, "\sigma", red] & (G_B^\inv, \cdot) \arrow[d, "\sigma", red] \\ 
(G_B^\inv, \overline\boxcon) \arrow[dr, "S^r"] \arrow[r, "S^r", red] & (G_B^\inv, \sqdot') \arrow[r, "\Id", blue] \arrow[d, "\Id"] & (G_B^\inv, \overline\cdot)\\ 
& (G_B^\inv, \star_r) \arrow[ur, "\Id", blue] &
\end{tikzcd}
\]

Note: The product $\overline \boxcon$ is the opposite product for $\boxcon$, it should not be confused with the product $\boxconvred$ from reference \cite{part1}.


\end{document}